\documentclass[10pt, a4paper]{article}

\usepackage{amsmath, amssymb, amsthm, mathtools, bbm}
\usepackage{mathrsfs}
\usepackage{url}
\usepackage{tikz}
\usepackage{color}
\usepackage{authblk}

\theoremstyle{plain}
\newtheorem{theorem}{Theorem}[section]
\newtheorem{proposition}[theorem]{Proposition}
\newtheorem{lemma}[theorem]{Lemma}

\newtheorem{fact}[theorem]{Fact}

\theoremstyle{definition}
\newtheorem{definition}{Definition}[section]
\newtheorem*{notation}{Notation}

\theoremstyle{remark}
\newtheorem{remark}{Remark}[section]

\title{ Scaling Limits of Density Fluctuation Fields of the Symmetric Partial Exclusion Processes Approaching the Zero Range Process}
\author{Kanta NOZAWA}
\affil{Graduate School of Mathematical Sciences, The University of Tokyo, 3-8-1 Komaba, Meguro-ku, Tokyo 153-8914, Japan}
\date{\today} 

\begin{document}

\maketitle

\begin{abstract}
  We study equilibrium fluctuations for the symmetric partial exclusion process (SPEP) on the one-dimensional integer lattice $\mathbb{Z}$, an interacting particle system in which each site is occupied by at most $\alpha$ particles. We assume that the jump rate is decomposable and that the instantaneous current admits a gradient form. Our aim is to study the scaling limit of the density fluctuation field of the SPEP starting from an invariant measure and to derive the corresponding stochastic partial differential equation (SPDE). 
  A distinctive feature of our setting is that the maximum number of particles per site depends on the scaling parameter $n$ as $\alpha=\alpha_n \propto n^a$ for some $a>0$, so that the model interpolates between the partial exclusion process ($\alpha <\infty$) and the symmetric zero range process (SZRP, $\alpha = \infty$). Under the diffusive rescaling in which the lattice spacing is $1/n$ and time is accelerated by a factor $n^2$, we prove that the density fluctuation field converges to the solution of the stochastic heat equation. 
  The limiting SPDE coincides with the one obtained for the SZRP. A key ingredient in the proof is the convergence, under a suitable rescaling, of the partition function of the reversible measure of the SPEP to that of the SZRP. This convergence implies that expectations of a wide class of local functions under the reversible measure of the SPEP converge to the corresponding expectations under the reversible measure of the SZRP, and provides a rigorous sense in which the SPEP approximates the SZRP as $n\to\infty$.
\end{abstract}

\section{Introduction}
  In non-equilibrium statistical mechanics and probability theory, the macroscopic behavior of interacting particle systems has been a subject of extensive study. A central result in this direction is the hydrodynamic limit, which asserts that the empirical measure associated with the particle density converges to the solution of a deterministic partial differential equation under a suitable space-time rescaling. This can be viewed as the law of large numbers for the particle system. In the present paper we consider fluctuations around this deterministic limit when the system is in equilibrium. More precisely, we study the density fluctuation field, which captures the deviation of the empirical particle density from its mean at each site, and establish that it converges, after some scaling, to the solution of a stochastic partial differential equation (SPDE).
 
  For a broad class of reversible systems with a conserved quantity on the one-dimensional integer lattice or the one-dimensional discrete torus, it is well-known that the scaling limit of the fluctuation field is the solution of the stochastic heat equation (SHE) \cite{BroxRost, MPSD1986, CCChung, Kipnis_Landm}:
\begin{equation}\label{SHE}
  du=D\Delta u\,dt + \sqrt{2\chi D}\,\nabla d\mathcal{B}_t.
\end{equation}
Here $D$ and $\chi$ are positive constants, and $\mathcal{B}$ denotes an $\mathcal{S}'(\mathbb{R})$-valued Brownian motion whose precise definition is given below.

Two representative examples of reversible systems whose density fluctuation field converges to the SHE are the symmetric simple exclusion process (SSEP) and the symmetric zero range process (SZRP). It is then natural to ask whether there exists an interacting particle system that interpolates between these two models, and whether the SHE persists as the scaling limit of its fluctuation field along the interpolation. Motivated by this question, in the present paper we consider the symmetric partial exclusion process (SPEP), an interacting particle system with exclusive interaction in which the maximum number of particles per site is $\alpha \in \mathbb{N}
$. The SPEP formally interpolates between the SSEP and the SZRP \cite{d80477dd344645dc8ca9cb717946fc90}: indeed, $\alpha = 1$ recovers the SSEP, while $\alpha = \infty$ formally corresponds to the SZRP.
The main result of \cite{d80477dd344645dc8ca9cb717946fc90} establishes that, for the SPEP satisfying the gradient condition with decomposable jump rates, the scaling limit of the density fluctuation field under the equilibrium measure is given by the stationary solution of the SHE. 
A key feature distinguishing our setting from \cite{d80477dd344645dc8ca9cb717946fc90} is that we let the jump rate depend on $n$ by taking $\alpha = \alpha_n \propto n^a$, so that the strength of the interaction between particles changes as $n \to \infty$. A primary goal of this paper is to clarify how the SPEP is connected to the SZRP in this limiting procedure.

Another motivation for introducing models with $n$-dependent jump rates comes from studies of the derivation of the stochastic Burgers equation (SBE): for some constant $\lambda \in \mathbb{R}$,

\begin{equation}\label{SBE}
  du=D\Delta udt + \frac{\lambda}{2}\nabla\!\left(u^2 \right)dt + \sqrt{2\chi D}\nabla d\mathcal{B}_t.
\end{equation}

It is well-known that the SBE has been derived from weakly asymmetric particle systems \cite{bertini1997stochastic, Gon_alves_2013, d80477dd344645dc8ca9cb717946fc90}, and a natural question is whether it can also be obtained from totally asymmetric systems. This has motivated the introduction of models in which the jump rates depend on the scaling parameter $n$. This idea was realized, for instance, in the totally asymmetric ZRP (TAZRP) with $n$-dependent jump rates \cite{HAYASHI2023180}, and more recently extended to the PEP with $\alpha$ depending on $n$ \cite{HS2026+} (in preparation). In \cite{HAYASHI2023180}, the TAZRP with $n$-dependent jump rates is studied. The key observation therein is that, heuristically, as $n \to \infty$ the interaction vanishes and the system converges to a system of independent random walks, equivalent to the SZRP without any interaction. Exploiting this structure, it is shown that the scaling limit of the fluctuation field of the TAZRP converges to the SBE.

We now return to the derivation of the SHE for the SPEP, which is the main concern of this paper. A distinctive feature of our setting, which arises from the $n$-dependence of $\alpha$, is that the invariant measure is of product form with marginals depending on $n$. Consequently, a precise understanding of the asymptotic behavior of expectations of local functions under this invariant measure is indispensable for the derivation of the scaling limit of the fluctuation field.

To address this, we show that the partition function of the reversible measure of the SPEP converges, under an appropriate rescaling of the partition function and its associated parameters, to that of the SZRP. As a consequence, expectations of a wide class of local functions under the reversible measure of the SPEP converge to the corresponding expectations under the reversible measure of the SZRP. This convergence plays a pivotal role in the proof of our main result and, moreover, provides a justification of the sense in which the SPEP approximates the SZRP.

\subsection{Organization of the Paper}
The rest of this paper is organized as follows. In Section 2, we introduce the notation and define the SPEP, and discuss its relation to the SZRP. In Section 3, we define the density fluctuation field, formulate the SHE via a martingale problem, and state our main result. The proof of the main theorem is carried out in Sections 4–6: Section 4 establishes the convergence of expectations under the reversible measure of the SPEP to those under the reversible measure of the SZRP, Section 5 is devoted to the first-order Boltzmann–Gibbs principle, and Section 6 completes the proof via tightness and the identification of limit points.
Basic notions concerning Schwartz spaces, tempered distributions, white noise, and distribution-valued Brownian motions are recalled in Appendix A.
\section{Model}\label{sec:model}
In this section, we introduce the symmetric partial exclusion process (SPEP) and its reversible measures.
\subsection{Notation}
Throughout this manuscript, let $T>0$, $\mathbb{N}\coloneq \{1,2,3,\cdots\}$, $\mathbb{N}_0\coloneq \{0,1,2,\cdots\}$, $n\in\mathbb{N}$. Furthermore, we fix $a, \alpha>0$ and let $\alpha_n\coloneq \lfloor{\alpha n^a}\rfloor$ where $\lfloor\cdot\rfloor$ denotes the integer part of the argument. In addition, $C>0$ denotes a generic constant whose value may change from line to
line. In this manuscript, $C$ is assumed not to depend on $n$, $l$ and $t$ (but may
depend on $f$, $\phi$, $T$ or any other parameters).

\subsection{The SPEP}
  Let $\Omega_n\coloneq\{0,1,\cdots,\alpha_n\}^\mathbb{Z}$ be the configuration space. An element of $\Omega_n$ is often denoted by $\eta=(\eta_x)_{x\in\mathbb{Z}}$.

  Suppose that a function $c:\mathbb{N}_0\to\mathbb{R}_{\geq0}$ satisfies the following four properties:
  \begin{itemize}
    \item $c(0)=0$
    \item $c$ is strictly increasing.
    \item $\liminf_{m\to\infty}c(m)/m \geq K$ for some $K>0$. In particular, $\lim_{m\to\infty}c(m)=\infty$.
    \item For any $\lambda>0$, the following series converges.
    \begin{equation}\label{fml-c-L-2}
      \sum_{m=2}^{\infty}\frac{c(m)\lambda^m}{c(m-1)c(m-2)\cdots c(1)}<\infty.
    \end{equation}
  \end{itemize}
  The last condition is technical, but it is necessary to ensure that $c$ is square-integrable with respect to the stationary measure of the ZRP. For example, $c(m)=\exp(m^\gamma)$ for some $\gamma>0$ satisfies this condition, while $c(m)=2^{2^m}$ does not.

  \begin{definition}
    Let $d_n:\{0,1,\cdots,\alpha_n\}\to\mathbb{R}_{\geq0}$ be a mapping with $d_n(0)\coloneq 0$ and $d_n(m)>0$ for $m\in\{1,2,\cdots,\alpha_n\}$. A Markov process on the state space $\Omega_n$ is called the symmetric partial exclusion process (or the SPEP for short) if its generator is given by
    \begin{equation}
      L_nf(\eta) \coloneq \frac{1}{c(\alpha_n)}\sum_{x,y\in\mathbb{Z}\,;\,|x-y|=1}c(\eta_x)d_n(\alpha_n-\eta_y)\left[f(\eta^{x,y})-f(\eta) \right]
    \end{equation}
    where $\eta\in\Omega_n$, $f:\Omega_n\to\mathbb{R}$ is a local function\footnote{We say that a function $f:\Omega_n\to\mathbb{R}$ is  \textit{local} if it depends only on a finite number of occupation variables, i.e.  if there exists a finite set $A \subset \mathbb Z$ such that $f (\eta) = f (\tilde\eta )$ for any
$\eta,\tilde\eta\in\Omega_n$ with $\eta(x) = \tilde \eta(x)$ for any $x \in A$. For a local function $f$, the support of $f$, denoted by $\mathrm{supp}(f)$, is defined as the smallest such set.}  and 
    \begin{equation}
      \eta^{x,y}(z)\coloneq
      \begin{cases}
        \eta_x-1\quad &z=x\\
        \eta_y+1\quad &z=y\\
        \eta_z\quad &\text{otherwise}.
      \end{cases}
    \end{equation}
  \end{definition}
  \begin{remark}
    The SPEP was introduced in \cite{d80477dd344645dc8ca9cb717946fc90}. In our setting, where the maximum number of particles per site depends on $n$, we need to slow down the dynamics defined in \cite{d80477dd344645dc8ca9cb717946fc90} by a factor of $c(\alpha_n)$ in order to obtain a non-trivial limit of the density fluctuation field as $n\to\infty$.
  \end{remark}

  \subsection{Stationary measures}
  \begin{definition}
    For $m\in\mathbb{N}_0$, define
    \begin{align*}
      c!(m) &\coloneq \begin{cases}
        c(m)c(m-1)\cdots c(1)\quad &m\in\mathbb{N} \\
        1\quad&m=0
      \end{cases}\\
      d_n!(m) &\coloneq \begin{cases}
        d_n(m)d_n(m-1)\cdots d_n(1)\quad &m\in\{1,2,\cdots,\alpha_n\} \\
        1\quad&m=0.
      \end{cases}\\
    \end{align*}
  \end{definition}
  \begin{definition}
    For $\lambda\in[0,\infty)$, we define a product probability measure $\overline{\nu}^{(n)}_\lambda$ on $\Omega_n$ with its common marginal
    \begin{equation}
      \overline{\nu}^{(n)}_\lambda\!\left(\eta_x=m \right)\coloneq \frac{1}{Z^{(n)}_\lambda}\frac{d_n!(\alpha_n)}{c!(m)d_n!(\alpha_n-m)}\lambda^m \quad\text{for }\, m\in\{0,1,\cdots,\alpha_n\},
    \end{equation}
  where
  \begin{equation}
    Z^{(n)}_\lambda\coloneq \sum_{m=0}^{\alpha_n} \frac{d_n!(\alpha_n)}{c!(m)d_n!(\alpha_n-m)}\lambda^m.
  \end{equation}
  \end{definition}

  The following proposition follows from a direct computation.
  \begin{proposition}
    The SPEP is reversible w.r.t. $\overline{\nu}^{(n)}_\lambda$.
  \end{proposition}

  Let us define $h_n(\lambda)\coloneq E_{\overline{\nu}^{(n)}_\lambda}\left[\eta_x \right]$, which represents the expected value of the number of particles per site.
  The following proposition shows that the reversible measures can be parametrized by density $h_n(\lambda)$.

  Since the following proposition can be verified by a direct computation, we omit the proof.

  \begin{proposition}
  The function $h_n$ is a strictly increasing $C^\infty$ function of $\lambda$, $h_n(0)=0$ and $\lim_{\lambda\to\infty}h_n(\lambda)=\alpha_n$ for each $n$.
  \end{proposition}

  Thus we can define $\nu^{(n)}_\rho\coloneq \overline{\nu}^{(n)}_{h^{-1}_n(\rho)}$ for $\rho\in[0,\alpha_n)$. Furthermore, let $\mathbb{P}_n$ be the underlying probability measure of the SPEP with initial distribution $\nu^{(n)}_\rho$ and let $\mathbb{E}_n$ denote the corresponding expectation.

  \subsection{Gradient condition}
  To proceed with our analysis, we introduce the instantaneous current associated with the SPEP. 
  For each $x\in\mathbb{Z}$, define the instantaneous current $W^{(n)}_{x,x+1}$ across the bond $\{x,x+1\}$ by
  \begin{equation}
    L_n\eta_x
    =
    W^{(n)}_{x-1,x}-W^{(n)}_{x,x+1}.
  \end{equation}
  A direct computation using the definition of the generator $L_n$ yields
  \begin{equation}
    W^{(n)}_{x,x+1}(\eta)
    =
    \frac{
      c(\eta_x)d_n(\alpha_n-\eta_{x+1})
      -
      c(\eta_{x+1})d_n(\alpha_n-\eta_x)
    }{c(\alpha_n)}.
  \end{equation}
  We now introduce the gradient condition.

  \begin{definition}
    We say that the SPEP is of gradient type if there exists a function
    $h\colon\{0,1,\ldots,\alpha_n\}\to\mathbb{R}$ such that
    \begin{equation}
      W^{(n)}_{x,x+1}(\eta)
      =
      h(\eta_x)-h(\eta_{x+1})
    \end{equation}
    for every $x\in\mathbb{Z}$ and $\eta\in\Omega_n$.
  \end{definition}
  As shown in \cite{d80477dd344645dc8ca9cb717946fc90}, the SPEP is of gradient type if and only if $d_n(m)$ is proportional to $c(\alpha_n)-c(\alpha_n-m)$ for $m\in\{0,1,\cdots,\alpha_n\}$. In particular, we only consider the case $d_n(m)=c(\alpha_n)-c(\alpha_n-m)$ in the sequel.
  In this case, the function $h$ in the gradient condition can be chosen as $h=c$.
  
  \subsection{Connection to the SZRP}
  The SZRP is an interacting particle system in which particles interact only with those occupying the same site. The SZRP is a Markov process on $\Omega_\infty\coloneq \mathbb{N}_0\,^{\mathbb{Z}}$ whose generator is given by
  \begin{equation}
    L_{\infty}f(\eta)\coloneq \sum_{x,y\in\mathbb{Z}\,;\,|x-y|=1}c(\eta_x)\left[f(\eta^{x,y})-f(\eta) \right].
  \end{equation}
  For the construction of the SZRP with non-decreasing jump rate $c$ in infinite volume, we refer to Section 2.6 in \cite{Kipnis_Landm} and \cite{BRSS}. 
  
  It is well-known that there exists a family of reversible product probability measures $\{\overline{\nu}_\lambda^{(\infty)}\}_{\lambda>0}$ whose marginal is
  \begin{align*}
    \overline{\nu}_\lambda^{(\infty)}\!\left(\eta_x=m \right)&=\frac{1}{Z_\lambda^{(\infty)}}\frac{\lambda^m}{c!(m)}\\
   Z_\lambda^{(\infty)}&\coloneq \sum_{m=0}^{\infty} \frac{\lambda^m}{c!(m)}.\\
  \end{align*}
  According to Section 2 of \cite{Kipnis_Landm}, if $\lim_{\lambda\to\lambda^*} Z_\lambda^{(\infty)}=\infty$ where $\lambda^*\in(0,\infty]$ is the convergence radius of $Z_\lambda^{(\infty)}$, then there exists a unique $\lambda^{(\infty)}_\rho>0$ such that $E_{\overline{\nu}^{(\infty)}_{\lambda^{(\infty)}_\rho}}\![\eta_x]=\rho$ for any $\rho>0$. We define $\nu^{(\infty)}_\rho\coloneq\overline{\nu}^{(\infty)}_{\lambda^{(\infty)}_\rho}$. In our case, we have $\lambda^*=\infty$.

  At a heuristic level, we can see that the SPEP approaches the SZRP as $n\to\infty$. Indeed, for any local function $f:\Omega_\infty\to\mathbb{R}$, we have
  \begin{align*}
    L_nf(\eta)= &\sum_{x,y\in\mathbb{Z}\,;\,|x-y|=1}c(\eta_x)\left(1-\frac{c(\eta_y)}{c(\alpha_n)} \right)\left[f(\eta^{x,y})-f(\eta) \right] \\ &\to \sum_{x,y\in\mathbb{Z}\,;\,|x-y|=1}c(\eta_x)\left[f(\eta^{x,y})-f(\eta) \right] \quad= L_\infty f(\eta).
  \end{align*}
  In the above formula, we regard $f$ as a function on $\Omega_n$ by restricting its domain to $\Omega_n\subset \Omega_\infty$. 
  
  \section{Main result}\label{sec-main-result}
  \begin{notation}
  Let us introduce some notation. Firstly, set $\chi^{(n)}(\rho)\coloneq \operatorname{Var}_{\nu^{(n)}_\rho}(\eta_x)$ for any $n\in\mathbb{N}\cup \{\infty\}$.

  Suppose that  $f\in L^1(\nu^{(\infty)}_\rho)$ is a local function (hence $f\in L^1(\nu^{(n)}_\rho)$ as well).
    \begin{itemize}
      \item $\Phi_f^{(n)}(\rho)\coloneq E_{\nu_\rho^{(n)}}[f]$
      \item $\Phi_f^{(\infty)}(\rho)\coloneq E_{\nu_\rho^{(\infty)}}[f]$
      \item $D_f^{(n)}(\rho)\coloneq \frac{d}{d\rho}\Phi_f^{(n)}(\rho)$
      \item $D_f^{(\infty)}(\rho)\coloneq \frac{d}{d\rho}\Phi_f^{(\infty)}(\rho)$
    \end{itemize}
    Furthermore, $\{\eta(t)\}_{t\in[0,T]}$ denotes a Markov process generated by $n^2L_n$ with initial distribution $\nu^{(n)}_\rho$. This is the SPEP sped up by $n^2$. Here, with a slight abuse of notation, we omit the dependence on $n$.

    In addition, let $\mathcal{B}_t$ be an $\mathcal{S}'(\mathbb{R})$-valued Brownian motion which is defined in Appendix \ref{sec-SHE}.
  \end{notation}
  Before stating the main result, let us define the density fluctuation fields and the notion of a solution to the SHE via a martingale problem.
  \begin{definition}
    We say that a stochastic process $\mathcal{X}^{(n)}\in D([0,T]:\mathcal{S}'(\mathbb{R}))$ is the density fluctuation field associated with the SPEP if for any $\phi\in\mathcal{S}(\mathbb{R})$, it holds that
    \begin{equation}
    \mathcal{X}^{(n)}_t(\phi)\coloneq \frac{1}{\sqrt{n}}\sum_{x\in\mathbb{Z}}\left(\eta_x(t)-\rho \right)\phi\!\left(\frac{x}{n}\right).
  \end{equation}
  \end{definition}
  \begin{definition}\label{def-sol-SHE}
    An $\mathcal{S}'(\mathbb{R})$-valued continuous process $\{\mathcal{Y}_t\}_{t\in[0,T]}$ is a solution to the SHE \eqref{SHE} if for any $\phi\in C([0,T]:\mathcal{S}(\mathbb{R}))\cap C^1((0,T):\mathcal{S}(\mathbb{R}))$,
    \begin{equation}
      \mathcal{Y}_t(\phi_t)-\mathcal{Y}_0(\phi_0)-\int_0^t\mathcal{Y}_s((\partial_s+D\Delta)\phi_s)\,ds
    \end{equation}
    is a continuous martingale of quadratic variation
    \begin{equation}
      2\chi D \int_0^t\langle\nabla\phi_s,\nabla\phi_s\rangle_{L^2(\mathbb{R})}ds,
    \end{equation}
    where $\phi_t=\phi(t,\cdot)$.
  \end{definition}
  It can be checked that the solution to the SHE is unique in law by using Itô's representation theorem (see \cite[Theorem 8.2]{DaPratoZabczyk} for details).

  Here is our main theorem of this manuscript. Note that $c\in L^1(\nu^{(\infty)}_\rho)$ by the assumption \eqref{fml-c-L-2}, in particular $D_c^{(\infty)}(\rho)$ and $\Phi_c^{(\infty)}(\rho)$ appearing below are well-defined.
  
  \begin{theorem}\label{Main-theorem}
  Fix a constant $\rho>0$ arbitrarily and suppose that the initial distribution of $\{\eta(t)\}_{t\in[0,T]}$ is $\nu^{(n)}_\rho$. Then the sequence $\{\mathcal{X}^{(n)}_\cdot \}_{n}$ converges in distribution on $D([0,T]:\mathcal{S}'(\mathbb{R}))$ to the solution of the SHE
  \begin{equation}\label{limit-eqn}
    d\mathcal{Y}_t= D_c^{(\infty)}(\rho)\Delta\mathcal{Y}_tdt + \sqrt{2\Phi_c^{(\infty)}(\rho)}\nabla d\mathcal{B}_t.
  \end{equation}
  \end{theorem}
  \begin{remark}
    In the case $a=0$ (i.e. $\alpha_n$ does not depend on $n$), $\Phi_c^{(n)}(\rho)$, $D_c^{(n)}(\rho)$, and $\chi^{(n)}(\rho)$ do not depend on $n$. As shown in \cite{d80477dd344645dc8ca9cb717946fc90}, the scaling limit of fluctuation fields is also the SHE
    \begin{equation}\label{GHS-limit-eq}
      d\mathcal{Y}_t= D_c^{(n)}(\rho)\Delta\mathcal{Y}_tdt + \sqrt{2D_c^{(n)}(\rho)\chi^{(n)}(\rho)}\nabla d\mathcal{B}_t.
    \end{equation}
    We can check that $\Phi_c^{(n)}(\rho)=D_c^{(n)}(\rho)\chi^{(n)}(\rho)$ for any $n\in\mathbb{N}\cup \{\infty\}$ by a direct computation. Therefore, comparing \eqref{limit-eqn} and \eqref{GHS-limit-eq}, we observe that our model interpolates between the SPEP with $a=0$ and the SZRP at the macroscopic level.
  \end{remark}

  \section{Preliminaries}\label{section-phi-D}
  This section is devoted to several auxiliary lemmas needed for the proof of our main result.
  Most of the notation used below has been introduced in Section \ref{sec:model}, to which we refer the reader.
  \subsection{Estimates related to the ZRP}
  In this subsection, we collect several estimates concerning the ZRP that will be used in the proof of the main theorem.
  \begin{lemma}\label{lemma-cvgc-part-fct}
    The following two statements hold:
    \begin{itemize}
      \item $Z_\lambda^{(n)}$ and $Z_\lambda^{(\infty)}$ are continuous and strictly increasing in $\lambda$.
      \item $\lim_{n\to\infty} Z_{\lambda/c(\alpha_n)}^{(n)}=Z_{\lambda}^{(\infty)} \quad \text{uniformly on compact sets}.$
    \end{itemize}
  \end{lemma}
  \begin{proof}
    The first item is immediate.
    For the second item, 
    \begin{align*}
    Z_{\lambda/c(\alpha_n)}^{(n)}&= \sum_{m=0}^{\alpha_n} \frac{1}{c!(m)}\prod_{i=0}^{m-1}\left(1-\frac{c(i)}{c(\alpha_n)} \right) \lambda^m \\ &\to \sum_{m=0}^{\infty} \frac{1}{c!(m)}\lambda^m =Z_{\lambda}^{(\infty)} \quad \text{as }\, n\to\infty. \\
    \end{align*}
    By using the Weierstrass M-test, we can show that the convergence above is uniform on compact sets.
  \end{proof}

  For notational convenience, define $F_n(\lambda)\coloneq h_n(\lambda/c(\alpha_n))$ for $n\in\mathbb{N}$, and $F_\infty(\lambda)\coloneq h_\infty(\lambda)$. 
  \begin{lemma}\label{lemma-cvgc-ev}
    
    For any $n\in\mathbb{N}\cup\{\infty\}$, $F_n(\cdot)$ is continuous and strictly increasing. In addition, $\lim_{n\to\infty}F_n(\lambda)=F_\infty(\lambda)$ uniformly on compact sets.
  \end{lemma}
  \begin{proof}
    For the proof that $F_\infty( \lambda)$ is continuous and strictly increasing, see \cite[Section 3]{Kipnis_Landm}.

    Next we consider $F_n(\lambda)$. A straightforward computation yields
    \begin{equation}
      \partial_\lambda F_n\!\left(\lambda \right)=\frac{\operatorname{Var}_{\overline{\nu}^{(n)}_\lambda} \!\left(\eta_x\right)}{\lambda}>0.
    \end{equation}
    Thus $F_n(\lambda)$ is strictly increasing in $\lambda$ and the continuity is obvious.

    Finally, let us prove that $\lim_{n\to\infty}F_n(\lambda)=F_\infty(\lambda)$ uniformly on compact sets. By the definition of $F_n(\lambda)$, 
    \begin{equation}
      F_n(\lambda)=\frac{1}{Z_{\lambda/c(\alpha_n)}^{(n)}}\sum_{m=0}^{\alpha_n}\frac{m}{c!(m)}\prod_{i=0}^{m-1}\left(1-\frac{c(i)}{c(\alpha_n)} \right) \lambda^m.
    \end{equation}
    Therefore we have the statement combining Lemma~\ref{lemma-cvgc-part-fct} and the Weierstrass M-test.
  \end{proof}

  Thanks to Lemma \ref{lemma-cvgc-ev}, there exists a nonnegative number $\lambda^{(n)}_\rho$ $( n\in\mathbb{N}\cup\{\infty\} )$ satisfying 
  \begin{equation}
      F_n(\lambda^{(n)}_\rho) = \rho,
  \end{equation}
  for any $\rho>0$ with $\alpha_n>\rho$. 

  \begin{lemma}\label{lemma-bdd-parameter}
    $\{\lambda^{(n)}_\rho \mid n\in\mathbb{N}\cup\{\infty\}\text{ with }\alpha_n>R,\,\, \rho\in[0,R]\}$ is a bounded subset of $\mathbb{R}$.
  \end{lemma}
  \begin{proof}
    Since $\lambda^{(n)}_\rho$ is increasing in $\rho$, it suffices to show that $\{\lambda^{(n)}_R\,:\alpha_n>R \}$ is a bounded subset.

    \begin{align*}
      \lim_{n\to\infty} F_n\!\left(\lambda^{(\infty)}_{R+1} \right)=F_\infty\!\left(\lambda^{(\infty)}_{R+1}\right)=R+1.
    \end{align*}
    Hence there exists sufficiently large $n_R\in\mathbb{N}$ such that $F_n(\lambda^{(\infty)}_{R+1})>R+\frac{1}{2}$ for any $n\geq n_R$. Therefore we have
    \begin{align*}
      \sup_{n\in\mathbb{N}}\lambda^{(n)}_R\leq \max_{n\leq n_R}\lambda^{(n)}_R \vee \lambda^{(\infty)}_{R+1},
    \end{align*}
    which concludes the proof.
  \end{proof}

  \begin{lemma}\label{lemma-cvgc-parameter}
    $\lim_{n\to\infty}\lambda^{(n)}_\rho=\lambda^{(\infty)}_\rho$\quad \text{uniformly on compact sets}.
  \end{lemma}
  \begin{proof}
    Since $\lambda^{(n)}_\rho$ is increasing in $\rho$, it is enough to check the pointwise convergence. 
    Moreover, since we have Lemma \ref{lemma-cvgc-ev}, it suffices to see
    \begin{equation}
      \lim_{n\to\infty} F_\infty(\lambda^{(n)}_\rho) = F_\infty(\lambda^{(\infty)}_\rho) \quad(=\rho).
    \end{equation}
    By Lemma \ref{lemma-bdd-parameter}, $M_R\coloneq\sup\{\lambda^{(n)}_\rho \mid n\in\mathbb{N}\cup\{\infty\}\text{ with }\alpha_n>R,\,\, \rho\in[0,R]\}<\infty$. Thus
    \begin{align*}
      |F_\infty(\lambda^{(n)}_\rho)-\rho|&= |F_\infty(\lambda^{(n)}_\rho)-F_n(\lambda^{(n)}_\rho)|\\
      &\leq \sup_{\lambda\in[0,M_R]} |F_\infty(\lambda)-F_n(\lambda) |.
    \end{align*}
    The rightmost term tends to 0 thanks to Lemma \ref{lemma-cvgc-ev}.
  \end{proof}

  \begin{lemma}\label{ev_cvgc}
    Let $f$ be a local function on $\Omega_\infty$ that belongs to $L^1(\overline{\nu}^{(\infty)}_{\lambda})$. Note that $f$ can also be regarded as a local function on $\Omega_n$, and as such belongs to $L^1(\overline{\nu}_{\lambda}^{(n)})$. Then, 
    \begin{align}\label{lem5-cvgc}
    \lim_{n\to\infty}\mathbb{E}_{\,\overline{\nu}^{(n)}_{\lambda}}\!\left[ f\right]=\mathbb{E}_{\,\overline{\nu}^{(\infty)}_{\lambda}}\!\left[ f \right].
    \end{align}
    Moreover, there exists a constant $C>0$ which does not depend on $n$ and $f$ (but may depend on $\lambda$) such that
    \begin{equation}\label{lem5-bound}
      \mathbb{E}_{\,\overline{\nu}^{(n)}_{\lambda}}\!\left[\lvert f\rvert \right]\leq C\mathbb{E}_{\,\overline{\nu}^{(\infty)}_{\lambda}}\!\left[\lvert f\rvert \right]\quad\text{uniformly on compact sets}.
    \end{equation}
  \end{lemma}

  \begin{proof}
    Without loss of generality we can assume that $f\geq0$ and $\text{supp}\,\,f\subset \{0,1,\ldots,l-1\}$ for a positive integer $l$.
    
    Let $I(n,m)\coloneq \prod_{i=0}^{m-1} \left[1-\frac{c(i)}{c(\alpha_n)} \right]$ for $m,n\in\mathbb{N}$ with $n\geq m$ and let
    \begin{align*}
      M\coloneq\sup_{n\in\mathbb{N}}\frac{Z^{(\infty)}_\lambda}{Z^{(n)}_\lambda}.
    \end{align*}
    This is finite since $Z^{(n)}_\lambda\geq1$. Then,
    \begin{equation}
      \begin{split}
        \mathbb{E}_{\,\overline{\nu}^{(n)}_{\lambda}}\!\left[ f \right]=\sum_{\eta_0=0}^{\alpha_n}\sum_{\eta_1=0}^{\alpha_n}\cdots\sum_{\eta_{l-1}=0}^{\alpha_n}f(\eta) \prod_{m=0}^{l-1} \frac{1}{Z^{(n)}_\lambda}\frac{1}{c!(\eta_m)}I(n,m) \lambda^m.
      \end{split}
    \end{equation}
    Here, 
    \begin{equation}\label{lem5-dom}
      0\leq f(\eta) \prod_{m=0}^{l-1} \frac{1}{Z^{(n)}_\lambda}\frac{1}{c!(\eta_m)}I(n,m) \lambda^m
      \leq Mf(\eta) \prod_{m=0}^{l-1} \frac{1}{Z^{(\infty)}_\lambda}\frac{1}{c!(\eta_m)} \lambda^m,
    \end{equation}
    and the dominated convergence theorem yields \eqref{lem5-cvgc}. Moreover, \eqref{lem5-dom} yields the bound \eqref{lem5-bound}. 
  \end{proof}

  \subsection{Estimates and convergence results of $\Phi_f^{(n)}(\rho)$ and $D_f^{(n)}(\rho)$}
  In this section, we establish the convergence of $\Phi_f^{(n)}(\rho)$ and $D_f^{(n)}(\rho)$ to $\Phi_f^{(\infty)}(\rho)$ and $D_f^{(\infty)}(\rho)$ respectively. These results are crucial in the proof of the main theorem.
  
  The following lemma follows from Lemma \ref{ev_cvgc} and the definition of $\nu^{(n)}_\rho$ and $\nu^{(\infty)}_\rho$.

  \begin{lemma}\label{lemma-cvgc-Phi}
    Let $f\in L^1(\nu^{(\infty)}_\rho)$ be a local function on $\Omega_\infty$. Then we have
    \begin{equation}
      \lim_{n\to\infty}\Phi^{(n)}_{f}(\rho)=\Phi^{(\infty)}_{f}(\rho).
    \end{equation}
  \end{lemma}
  
  We now turn to the convergence of $D^{(n)}_f(\rho)$. The following lemma plays a crucial role.

  \begin{lemma}\label{lem-chi-times-D-equal-Phi}
    Suppose that $l\in\mathbb{N}$ and $f$ is a local function on $\Omega_n$ with $\operatorname{supp} f\subset \{0,1,\cdots,l-1\}$. Then, for any $\rho\in (0,\alpha_n)$, we have
    \begin{equation}
      \chi^{(n)}(\rho)D^{(n)}_f(\rho)=\operatorname{Cov}_{\nu^{(n)}_\rho}\!\left(f, \sum_{x=0}^{l-1}\eta_x\right).
    \end{equation}
  \end{lemma}
  \begin{proof}
    For any $\lambda>0$ and any $m\in\{0,1,\cdots,\alpha_n\}$, we define
    \begin{equation}
      \overline{\Phi}^{(n)}_f(\lambda)\coloneq E_{\overline{\nu}^{(n)}_\lambda}\!\left[f \right]
    \end{equation}
    and
    \begin{equation}
      J(n,m)\coloneq \frac{d_n!(\alpha_n)}{c!(m)d_n!(\alpha_n-m)}.
    \end{equation}
    Then we have
    \begin{align}
      \overline{\Phi}^{(n)}_f(\lambda)=\sum_{\eta_0,\eta_1,\cdots,\eta_{l-1}=0}^{\alpha_n}f(\eta)\prod_{x=0}^{l-1}J(n,\eta_x)\frac{\lambda^{\eta_x}}{Z^{(n)}_\lambda}.
    \end{align}
    Differentiating both sides with respect to $\lambda$, we obtain
    \begin{align}
      \frac{d}{d\lambda}\overline{\Phi}^{(n)}_f(\lambda)&=\sum_{y=0}^{l-1}\sum_{\eta_0,\eta_1,\cdots,\eta_{l-1}=0}^{\alpha_n}f(\eta)\prod_{x\in\{0,1,\cdots,l-1\}\setminus\{y\}}J(n,\eta_x)\frac{\lambda^{\eta_x}}{Z^{(n)}_\lambda}\cdot J(n,\eta_y)\frac{d}{d\lambda}\left(\frac{\lambda^{\eta_y}}{Z^{(n)}_\lambda}\right)\\
      &=\sum_{y=0}^{l-1}\sum_{\eta_0,\eta_1,\cdots,\eta_{l-1}=0}^{\alpha_n}f(\eta)\prod_{x\in\{0,1,\cdots,l-1\}\setminus\{y\}}J(n,\eta_x)\frac{\lambda^{\eta_x}}{Z^{(n)}_\lambda}\cdot J(n,\eta_y)\frac{\eta_y Z^{(n)}_\lambda-\lambda \frac{d}{d\lambda}Z^{(n)}_\lambda}{(Z^{(n)}_\lambda)^2}\lambda^{\eta_y-1}\\
      &=\frac{1}{\lambda}\sum_{y=0}^{l-1}\left\{E_{\overline{\nu}^{(n)}_\lambda}\!\left[f\eta_y \right]-E_{\overline{\nu}^{(n)}_\lambda}\!\left[f \right]E_{\overline{\nu}^{(n)}_\lambda}\!\left[\eta_y \right] \right\}\\
      &=\frac{1}{\lambda}\sum_{y=0}^{l-1}\operatorname{Cov}_{\overline{\nu}^{(n)}_\lambda}\!\left(f,\eta_y \right).\\
      &= \frac{1}{\lambda}\operatorname{Cov}_{\overline{\nu}^{(n)}_\lambda}\!\left(f,\sum_{x=0}^{l-1}\eta_x \right).
    \end{align}
    In the third equality, we used the fact that $E_{\overline{\nu}^{(n)}_\lambda}\!\left[\eta_y \right]=\lambda \frac{d}{d\lambda}Z^{(n)}_\lambda/Z^{(n)}_\lambda$.
    Combining the above with the definition of $D^{(n)}_f(\rho)$, we obtain
    \begin{align}
      D^{(n)}_f(\rho)&=\frac{d}{d\rho}\Phi^{(n)}_f(\rho)\\
      &=\frac{d}{d\rho}\lambda^{(n)}_\rho\frac{d}{d\lambda}\overline{\Phi}^{(n)}_f(\lambda^{(n)}_\rho)\cdot\\
      &=\frac{\lambda^{(n)}_\rho}{\chi^{(n)}(\rho)}\cdot \frac{1}{\lambda^{(n)}_\rho}\operatorname{Cov}_{\overline{\nu}^{(n)}_{\lambda^{(n)}_\rho}}\!\left(f,\sum_{x=0}^{l-1}\eta_x \right)\\
      &=\frac{1}{\chi^{(n)}(\rho)}\operatorname{Cov}_{\nu^{(n)}_\rho}\!\left(f,\sum_{x=0}^{l-1}\eta_x \right).
    \end{align}
    Therefore we have the desired result.
  \end{proof}

  In the case $n=\infty$, we have the following lemma.
  \begin{lemma}\label{lem-chi-times-D-equal-Phi-infty}
    Suppose that $l\in\mathbb{N}$, $R>0$ and that $f\in L^1(\nu^{(\infty)}_R)$ is a local function with $\text{supp}f\subset\{0,1,\cdots,l-1\}$. Then, for any $\rho\in(0,R)$, 
    \begin{equation}
      \chi^{(\infty)}(\rho)D^{(\infty)}_f(\rho)=\operatorname{Cov}_{\nu^{(\infty)}_\rho}\!\left(f, \sum_{x=0}^{l-1}\eta_x\right).
    \end{equation}
  \end{lemma}
  \begin{proof}
  The proof is similar to that of Lemma \ref{lem-chi-times-D-equal-Phi}. We omit the details.

  However, we need to check that the covariance $\operatorname{Cov}_{\nu^{(\infty)}_\rho}\!\left(f, \sum_{x=0}^{l-1}\eta_x\right)$ is finite. In order to see this, it suffices to show that $\eta_xf\in L^1(\nu^{(\infty)}_\rho)$ for all $x\in\{0,1,\cdots,l-1\}$. Without loss of generality, we can assume that $f\geq0$. Then, by the definition of $\nu^{(\infty)}_\rho$, we have
  \begin{align}
    E_{\nu^{(\infty)}_\rho}\!\left[\eta_xf \right]&=\sum_{\eta_0,\eta_1,\cdots,\eta_{l-1}=0}^{\infty}\eta_xf(\eta)\prod_{y=0}^{l-1}\frac{1}{c!(\eta_y)}\frac{\lambda^{(\infty)\eta_y}_\rho}{\overline{Z}^{(\infty)}_\rho}\\
    &\leq \sum_{\eta_0,\eta_1,\cdots,\eta_{l-1}=0}^{\infty}f(\eta) \eta_x\left(\frac{\lambda^{(\infty)}_\rho}{\lambda^{(\infty)}_R} \right)^{\eta_x} \prod_{y=0}^{l-1}\frac{1}{c!(\eta_y)}\frac{\lambda^{(\infty)\eta_y}_R}{\overline{Z}^{(\infty)}_\rho}\\
    &\leq C(\rho,R) E_{\nu^{(\infty)}_R}\!\left[f \right]<\infty,
  \end{align}
  where $C(\rho,R)$ is a positive constant depending on $\rho$ and $R$.
  \end{proof}

  \begin{lemma}\label{lemma-cvgc-D}
    Suppose that $R>0$ and that $f\in L^1(\nu^{(\infty)}_R)$ is a local function. Then, for any $\rho\in(0,R)$, 
    \begin{equation}
      \lim_{n\to\infty}D_f^{(n)}(\rho)=D_f^{(\infty)}(\rho).
    \end{equation}
  \end{lemma}
  \begin{proof}
    Combining Lemmas \ref{ev_cvgc}, \ref{lem-chi-times-D-equal-Phi} and \ref{lem-chi-times-D-equal-Phi-infty}, we have the desired result. Note that $\chi^{(n)}(\rho)>0$ for any $n\in \mathbb{N}\cup\{\infty \}$.
  \end{proof}

  \section{The first-order Boltzmann--Gibbs principle}
  In this section, we prove the first-order Boltzmann–-Gibbs principle, which plays a crucial role in the proof of our main theorem.

  \begin{theorem}\label{BGP}
    For any local function $f\in L^2(\nu^{(\infty)}_\rho)$ and for any $\phi\in C([0,T]:\mathcal{S}(\mathbb{R}))$, 
    \begin{equation}\label{BGP-fml}
    \begin{split}
      \lim_{n\to\infty} \mathbb{E}_n\!\biggl[\,\sup_{t\in[0,T]}& \biggl|\int_0^t \frac{1}{\sqrt{n}}\sum_{x\in\mathbb{Z}}\bigl\{\tau_x f(\eta(s)) \\ 
      &- \Phi_f^{(n)}(\rho) - D^{(n)}_f\!(\rho)\left(\eta_x(s)-\rho \right) \bigr\}\phi\!\left(s,\frac{x}{n} \right)\,ds \biggr|^2 \biggr]= 0,
    \end{split}
    \end{equation}
    where
  \begin{align*}
    \tau_x f(\eta)&\coloneq f(\tau_x \eta)\\
    \tau_x \eta_y &\coloneq \eta_{x+y} \quad\text{for any}\,\,y\in\mathbb{Z}.
  \end{align*}
  \end{theorem}
  \subsection{Kipnis--Varadhan inequality}
  We first introduce Sobolev norms associated with the generator $L_n$. For any $h\in L^2(\nu^{(n)}_\rho)$, define $\|\cdot\|_{1,n}$ by
  \begin{equation}
    \|h\|_{1,n}^2 \coloneq \left<h,(-n^2L_n)h\right>_{\nu^{(n)}_\rho}.
  \end{equation}
  In addition, for $h\in L^2(\nu^{(n)}_\rho)$ with $E_{\nu^{(n)}_\rho}[h]=0$, define
  \begin{equation}\label{dual-sobolev}
    \|h\|^2_{-1,n}
    \coloneq
    \sup\left\{
      2\left<h,g\right>_{\nu^{(n)}_\rho}
      - n^2\left<g,(-L_n)g\right>_{\nu^{(n)}_\rho}
      \,:\,
      g\colon\Omega_n\to\mathbb{R}
      \text{ is a local function}
    \right\}.
  \end{equation}
  The following Kipnis--Varadhan inequality provides an estimate for additive
  functionals of the stationary process in terms of this norm.

  \begin{proposition}{(Kipnis--Varadhan inequality)}\label{KVI}
    Let $f\colon[0,T]\to L^2(\nu^{(n)}_\rho)$ with $E_{\nu^{(n)}_\rho}[f(s,\cdot)]=0$ for all $s\in[0,T]$.
    Then there exists a universal constant $C_0>0$ such that
    \begin{equation}\label{Kipnis--Varadhan}
      \mathbb{E}_n\!\left[
        \sup_{t\in[0,T]}
        \left|
          \int_0^t f(s,\eta(s))\,ds
        \right|^2
      \right]
      \leq
      C_0
      \int_0^T
        \|f(s,\cdot)\|^2_{-1,n}\,ds.
    \end{equation}
    In particular, we can take $C_0=27$.
  \end{proposition}
  \begin{proof}
    This proof is due to \cite[Lemma 4.3]{Chang2001} and \cite[Appendix A, Proposition 6.1]{Kipnis_Landm}. 
    By the definition of $\|\cdot\|_{-1,n}$, we have that for any $A>0$,
    \begin{equation}\label{estimate}
      \left<h,g\right>_{\nu^{(n)}_\rho}\leq \frac{A}{2}\|h\|_{1,n} + \frac{1}{2A}\|g\|_{-1,n}.
    \end{equation}
    Since $-n^2L_n$ is a nonnegative self-adjoint operator on $L^2(\nu^{(n)}_\rho)$, $\left(\gamma-n^2L_n\right)^{-1}$ is a bounded operator on $L^2(\nu^{(n)}_\rho)$ for any $\gamma>0$. Thus for any $\gamma>0$ and any $g\in L^2(\nu^{(n)}_\rho)$ with $\lVert g \rVert_{-1,n}<\infty$, there exists $h_\gamma\in L^2(\nu^{(n)}_\rho)$ such that 
    \begin{equation}\label{h_gamma}
      \gamma h_\gamma - n^2L_n h_\gamma = g.
    \end{equation}
    Taking $\left<\cdot,h_\gamma\right>$ on both sides and using the inequality \eqref{estimate} yield
    \begin{align*}
      \gamma \left<h_\gamma,h_\gamma \right>_{\nu^{(n)}_\rho} &\leq \|g\|_{-1,n}^2\\
      \|h_\gamma\|_{1,n}^2&\leq\|g\|_{-1,n}^2.
    \end{align*}
    Using the above formula, we have
    \begin{equation}
      \|g-\left(-n^2L_n\right)h_\gamma\|^2_{L^2(\nu^{(n)}_\rho)}\leq \gamma\|g\|_{-1,n}^2\to0\quad \text{as}\,\,\gamma\searrow0.
    \end{equation}
    Therefore $g$ can be approximated by elements of the range of the operator $L_n$ as long as $\|g\|_{-1,n}<\infty$. 
    
    Now let us move on to the proof of the Kipnis--Varadhan inequality. In order to prove \eqref{Kipnis--Varadhan}, we can assume that the right-hand side of \eqref{Kipnis--Varadhan} is finite. 
    Suppose that a function $f_\gamma(s,\cdot)$ satisfies
    \begin{equation}
      \gamma f_\gamma(s,\cdot) - n^2L_n f_\gamma(s,\cdot) = f(s,\cdot).
    \end{equation}
    Then we obtain
    \begin{equation}
      \begin{split}
        \mathbb{E}_n\!\left[\sup_{t\in[0,T]}\left|\int_0^t f(s,\eta(s))\,ds \right|^2 \right]&\leq 2\mathbb{E}_n\!\left[\sup_{t\in[0,T]}\left|\int_0^t f(s,\eta(s))+n^2L_nf_\gamma(s,\eta(s)) \,ds \right|^2 \right]\\
        &\quad\quad+2\mathbb{E}_n\!\left[\sup_{t\in[0,T]}\left|\int_0^t n^2L_nf_\gamma(s,\eta(s))\,ds \right|^2 \right]\\
        &\leq 2T\int_0^T \|f(s,\eta(s))+n^2L_nf_\gamma(s,\eta(s))\|_{L^2(\nu^{(n)}_\rho)}^2\,ds\\
        &\quad\quad +2\mathbb{E}_n\!\left[\sup_{t\in[0,T]}\left|\int_0^t n^2L_nf_\gamma(s,\eta(s))\,ds \right|^2 \right]\\
        &\leq 2T\gamma \int_0^T \|f(s,\cdot)\|^2_{-1,n}\,ds\\
        &\quad\quad +2\mathbb{E}_n\!\left[\sup_{t\in[0,T]}\left|\int_0^t n^2L_nf_\gamma(s,\eta(s))\,ds \right|^2 \right].\\
      \end{split}
    \end{equation}
    Since we can take $\gamma>0$ arbitrarily, the first term on the right-hand side will be arbitrarily small. Therefore we only have to take care of the second term.

    We can make use of the Dynkin martingale:
    \begin{equation}
      M_t(\eta)\coloneq f_\gamma(t,\eta(t))-f_\gamma(0,\eta(0))-\int_0^t \left(\partial_s+n^2L_n \right)f_\gamma(s,\eta(s))\,ds.
    \end{equation}
    This is a martingale with respect to the filtration generated by $\{\sigma(\eta(s))\,:\,s\in[0,t] \}$. On the other hand, we have so-called the backward martingale
    \begin{equation}
      M_t^-(\eta)\coloneq f_\gamma(T-t,\eta(T-t))-f_\gamma(T,\eta(T)) -\int_0^t \left(-\partial_s+n^2L_n^* \right)f_\gamma(T-s,\eta(T-s))\,ds
    \end{equation}
    In the above expression, $\partial_s$ denotes the derivative with respect to the first argument of the operand (not with respect to the variable $s$) and $L_n^*$ is the adjoint operator on $L^2(\nu^{(n)}_\rho)$ (in this case $L_n^*=L_n$). This is a $\{\sigma(\eta(s))\,:\,s\in[t,T]\}$-martingale.
    Changing variables in the integral in the backward martingale yields
    \begin{equation}
      M_T^--M_{T-t}^- = f_\gamma(0,\eta(0))-f_\gamma(t,\eta(t))-\int_0^t (-\partial_s+n^2L_n)f_\gamma(s,\eta(s))\,ds.
    \end{equation}
    Then
    \begin{equation}
      M_t+M_T^--M_{T-t}^-=2\int_0^t n^2L_n f_\gamma(s,\eta(s))\,ds.
    \end{equation}
    Therefore
    \begin{equation}\label{rewriting-LHS-KVineq}
      \mathbb{E}_n\!\left[\sup_{t\in[0,T]}\left|\int_0^t n^2L_nf_\gamma(s,\eta(s))\,ds \right|^2 \right]=\frac{1}{4}\mathbb{E}_n\!\left[\sup_{t\in[0,T]}\left(M_t+M_T^--M_{T-t}^- \right)^2 \right].
    \end{equation}
    By Doob's inequality and an elementary inequality $(a_1+a_2+a_3)^2\leq3(a_1^2+a_2^2+a_3^2)$, the right-hand side of \eqref{rewriting-LHS-KVineq} is bounded from above by
    \begin{equation}
      \mathbb{E}_n\!\left[3\left<M\right>_T + \frac{15}{4}\left<M^-\right>_T \right],
    \end{equation}
    where $\left<M \right>_\cdot$ denotes the quadratic variation. We can check that 
    \begin{equation}
    \begin{split}
      \mathbb{E}_n\left[\left<M \right>_T \right]&=\mathbb{E}_n\left[\left<M^- \right>_T \right] \\
      &= 2\int_0^T \|f_\gamma(s,\cdot)\|_{1,n}^2\,ds\\
      &\leq 2\int_0^T \|f(s,\cdot)\|_{-1,n}^2\,ds.\\
    \end{split}
    \end{equation}
    Thus we obtain
    \begin{equation}
      \mathbb{E}_n\!\left[\sup_{t\in[0,T]}\left|\int_0^t n^2L_nf_\gamma(s,\eta(s))\,ds \right|^2 \right]\leq \frac{27}{2}\int_0^T \|f(s,\cdot)\|_{-1,n}^2\,ds,
    \end{equation}
    which completes the proof.
  \end{proof}
  \subsection{The proof of Theorem \ref{BGP} }
  \subsubsection{Step 1: localization}
  Without loss of generality, we can assume that $\text{supp}\,f\subset \{0,1,\cdots,l_0-1\}$ where $l_0\in\mathbb{N}$. In addition, assume that $l\in\{3l_0+1,3l_0+2,\cdots\}$ and set
  \begin{equation}
    V^{(n)}(\eta)\coloneq f(\eta) - \Phi_f^{(n)}(\rho) - D^{(n)}_f\!(\rho)\left(\eta_0-\rho \right). 
  \end{equation}
  Furthermore, for each integer $j$, define
  \begin{align*}
    B_j&\coloneq lj+[0,l-1]\cap\mathbb{Z}\\
    B^0_j&\coloneq lj+[l_0, l-l_0 -1]\cap\mathbb{Z}\\
    B^1_j&\coloneq lj+[-l_0,l_0-1]\cap\mathbb{Z},\\
  \end{align*}
  thus $\{B^0_j,B^1_j\,:\, j\in\mathbb{Z}\}$ are disjoint subsets and $\mathbb{Z}=\bigcup \left(B^0_j\cup B^1_j \right)$.

  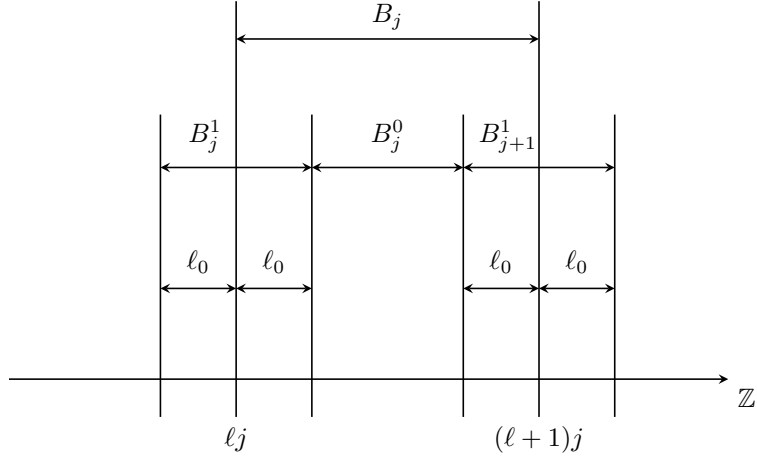
\begin{figure}[htbp]
  \centering
  \begin{tikzpicture}[>=stealth, semithick]

    \draw[->] (0, 0) -- (9.5, 0) node[below right] {$\mathbb{Z}$};

    \draw (3, -0.5) -- (3, 5);
    \node at (3, -0.8) {$\ell j$};

    \draw (7, -0.5) -- (7, 5);
    \node at (7, -0.8) {$(\ell+1) j$};

    \draw (2.0, -0.5) -- (2.0, 3.5);
    \draw (4.0, -0.5) -- (4.0, 3.5);
    \draw (6.0, -0.5) -- (6.0, 3.5);
    \draw (8.0, -0.5) -- (8.0, 3.5);

    \draw[<->] (3, 4.5) -- (7, 4.5) node[midway, above] {$B_j$};

    \draw[<->] (2.0, 2.8) -- (4.0, 2.8) node[midway, above=2pt] {$ B_j^1\quad\quad\,\, $};
    \draw[<->] (4.0, 2.8) -- (6.0, 2.8) node[midway, above=2pt] {$B_j^0$};
    \draw[<->] (6.0, 2.8) -- (8.0, 2.8) node[midway, above=2pt] {$B_{j+1}^1\quad\quad\,\,$};

    \draw[<->] (2.0, 1.2) -- (3, 1.2) node[midway, above=2pt] {$\ell_0$};
    \draw[<->] (3, 1.2) -- (4.0, 1.2) node[midway, above=2pt] {$\ell_0$};
    \draw[<->] (6.0, 1.2) -- (7, 1.2) node[midway, above=2pt] {$\ell_0$};
    \draw[<->] (7, 1.2) -- (8.0, 1.2) node[midway, above=2pt] {$\ell_0$};

\end{tikzpicture}
\caption{The integer lattice is divided into $\{B_j,B^0_j,B^1_j:j\in\mathbb{Z} \}$.}
\label{fig-Z-division}
\end{figure}

 \begin{align*}
   \mathbb{E}_n\!\left[\sup_{t\in[0,T]}\left|\int_0^t \frac{1}{\sqrt{n}}\sum_{x\in\mathbb{Z}}\tau_xV^{(n)}(\eta(s))\phi\!\left(s,\frac{x}{n} \right)\,ds\right|^2 \right]&\leq 2\left((\mathrm{I})+(\mathrm{II}) \right),\\
 \end{align*}
 where
 \begin{align*}
   (\mathrm{I}) &= \mathbb{E}_n\!\left[\sup_{t\in[0,T]}\left|\int_0^t \frac{1}{\sqrt{n}}\sum_{j\in\mathbb{Z}}\sum_{x\in B^0_j}\tau_xV^{(n)}(\eta(s))\phi\!\left(s,\frac{x}{n} \right)\,ds\right|^2 \right]\\
   (\mathrm{II}) &= \mathbb{E}_n\!\left[\sup_{t\in[0,T]}\left|\int_0^t \frac{1}{\sqrt{n}}\sum_{j\in\mathbb{Z}}\sum_{x\in B^1_j}\tau_xV^{(n)}(\eta(s))\phi\!\left(s,\frac{x}{n} \right)\,ds\right|^2 \right].
 \end{align*}
 It suffices to check that
 \begin{align}
   \limsup_{l\to\infty} \limsup_{n\to\infty} (\mathrm{I})&=0 \label{(I)} \\
   \limsup_{l\to\infty} \limsup_{n\to\infty} (\mathrm{II})&=0 \label{(II)}.
 \end{align}
 
  Firstly let us prove \eqref{(II)}. Indeed, by the Cauchy--Schwarz inequality, the cardinality of $B^1_j$ equals $2l_{0}$ and the fact that $\tau_xV^{(n)}$ and $\tau_yV^{(n)}$ are independent as long as $x\in B^1_i$ and $y\in B^1_j$ for distinct $i$ and $j$,
 \begin{align}
   (\mathrm{II})&\leq T\int_0^T \frac{1}{n}\sum_{j\in\mathbb{Z}}\,\sum_{x,y\in B^1_j} E_{\nu^{(n)}_\rho}\left[\tau_xV^{(n)}\cdot \tau_yV^{(n)} \right]\phi\!\left(s,\frac{x}{n} \right)\phi\!\left(s,\frac{y}{n} \right) \,ds\\
   &\leq C\int_0^T \frac{1}{n}\sum_{j\in\mathbb{Z}}\,\sum_{x\in B^1_j} \phi\!\left(s,\frac{x}{n} \right)^2\,ds\\
   &= C\int_0^T \frac{1}{n}\sum_{x\in B^1_0} \,\sum_{j\in\mathbb{Z}}\phi\!\left(s,\frac{lj+x}{n} \right)^2\,ds\\
   &\leq \frac{C}{n}\sum_{j\in\mathbb{Z}} \frac{1}{1+\left(\frac{lj}{n} \right)^2}\,ds\\
   &\leq C\left(\frac{1}{n} +\int_{-\infty}^{\infty} \frac{1}{1+l^2x^2}\,dx\right).
 \end{align}

 The last expression tends to 0 as $n\to\infty$ and then $l\to\infty$, hence we obtain \eqref{(II)}. Note that $E_{\nu^{(n)}_\rho}\left[\left( \tau_xV^{(n)}\right)^2 \right]$ is uniformly bounded in $n$.

 Moreover, let $y_j=lj$. Then

 \begin{equation}
   (\mathrm{I})\leq2\left((\mathrm{III})+(\mathrm{IV}) \right),
 \end{equation}
 where
 \begin{align*}
   (\mathrm{III})&= \mathbb{E}_n\!\left[\sup_{t\in[0,T]}\left|\int_0^t \frac{1}{\sqrt{n}}\sum_{j\in\mathbb{Z}}\sum_{x\in B^0_j}\tau_xV^{(n)}(\eta(s))\phi\!\left(s,\frac{y_j}{n} \right)\,ds\right|^2 \right]\\
   (\mathrm{IV})&= \mathbb{E}_n\!\left[\sup_{t\in[0,T]}\left|\int_0^t \frac{1}{\sqrt{n}}\sum_{j\in\mathbb{Z}}\sum_{x\in B^0_j}\tau_xV^{(n)}(\eta(s))\left|\phi\!\left(s,\frac{x}{n}\right)- \phi\!\left(s,\frac{y_j}{n}\right) \right|\,ds\right|^2 \right].\\
 \end{align*}
 Applying the Cauchy--Schwarz inequality to $(\mathrm{IV})$,
 \begin{equation}
 \begin{split}
   (\mathrm{IV})&\leq C\int_0^T \frac{1}{n}\sum_{j\in\mathbb{Z}}\,\sum_{x,y\in B^0_j} E_{\nu^{(n)}_\rho}\left[\tau_xV^{(n)}\cdot \tau_yV^{(n)} \right]\\ 
   &\quad\quad\quad\cdot\left|\phi\!\left(s,\frac{x}{n}\right)- \phi\!\left(s,\frac{y_j}{n}\right) \right| \left|\phi\!\left(s,\frac{y}{n}\right)- \phi\!\left(s,\frac{y_j}{n}\right) \right|  \,ds\\
   &\leq C \frac{l}{n} \int_0^T \frac{1}{n} \sum_{j\in\mathbb{Z}}\,\sum_{x\in B^0_j} \left(\left|\phi\!\left(s,\frac{x}{n}\right)\right|+ \left|\phi\!\left(s,\frac{y_j}{n}\right)\right|\right)\,ds\\
   &\leq C\frac{l^2}{n}.
 \end{split}
 \end{equation}
 In the second inequality, we used the fact that 
 \begin{align}
   \left|\phi\!\left(s,\frac{y}{n}\right)- \phi\!\left(s,\frac{y_j}{n}\right) \right| &\leq C\frac{|y-y_j|}{n} \leq C\frac{l}{n},\\
   \left|\phi\!\left(s,\frac{x}{n}\right)- \phi\!\left(s,\frac{y_j}{n}\right) \right| &\leq \left\lvert \phi\!\left(s,\frac{x}{n}\right)+ \phi\!\left(s,\frac{y_j}{n}\right) \right\rvert.
 \end{align}
  and that $E_{\nu^{(n)}_\rho}\left[\tau_xV^{(n)}\cdot \tau_yV^{(n)} \right]$ is uniformly bounded in $n$, $x$, and $y$. 
 In the third inequality, we used the fact that the cardinality of $B^0_j$ is less than $l$.
 Therefore, we have $\limsup_{l\to\infty}\limsup_{n\to\infty} (\mathrm{(IV)})=0$.

 \subsubsection{Step 2: approximation by elements of the range of $L^{(\infty)} _{B_j}$}
 Our remaining task is to prove the following:
 \begin{equation}{\label{III}}
   \limsup_{l\to\infty}\limsup_{n\to\infty} (\mathrm{III})=0.
 \end{equation}
 Let $L^{(n)} _{B_j}$ be a linear operator
 \begin{align*}
   L^{(n)}_{B_j}f(\eta)&\coloneq \sum_{x,y\in B_j\,;\,|x-y|=1}r_n(\eta_x,\eta_y)\left[f(\eta^{x,y})-f(\eta) \right]\\
   r_n(\eta_x,\eta_y)&\coloneq c(\eta_x)\left(1-\frac{c(\eta_y)}{c(\alpha_n)} \right),
 \end{align*}
 and 
 
 \begin{align*}
   L^{(\infty)}_{B_j}f(\eta)&\coloneq \sum_{x,y\in B_j\,;\,|x-y|=1}r_{\infty}(\eta_x,\eta_y)\left[f(\eta^{x,y})-f(\eta) \right]\\
   r_\infty(\eta_x,\eta_y)&\coloneq c(\eta_x).
 \end{align*}
 
 In addition, let $f_0$ be a bounded local function with $\text{supp}\,\,f_0\subset B_0$ and $f_j\coloneq \tau_{lj}f_0 $ (hence $\text{supp}\,\,f_j\subset B_j$). Then we have that
 \begin{equation}
   (\mathrm{III})\leq 2\left((\mathrm{V})+(\mathrm{VI}) \right),
 \end{equation}
 where
 \begin{align*}
   (\mathrm{V})&= \mathbb{E}_n\!\left[\sup_{t\in[0,T]}\left|\int_0^t \frac{1}{\sqrt{n}}\sum_{j\in\mathbb{Z}}\phi\!\left(s,\frac{y_j}{n} \right) \left(-L^{(n)}_{B_j}f_j(\eta(s))+\sum_{x\in B^0_j} \tau_xV^{(n)}(\eta(s)) \right)\,ds\right|^2 \right]\\
   (\mathrm{VI}) &= \mathbb{E}_n\!\left[\sup_{t\in[0,T]}\left|\int_0^t \frac{1}{\sqrt{n}}\sum_{j\in\mathbb{Z}}\phi\!\left(s,\frac{y_j}{n} \right) L^{(n)}_{B_j}f_j(\eta(s))\,ds\right|^2 \right].\\
 \end{align*}
 In order to complete the proof of Theorem \ref{BGP}, it suffices to prove that
 \begin{align}
   \limsup_{l\to\infty} \inf_{f_0\in L^2(\nu^{(\infty)}), local} \limsup_{n\to\infty} (\mathrm{V})&=0 \label{V}\\
   \limsup_{l\to\infty} \limsup_{n\to\infty} (\mathrm{VI})=0\label{VI}.
 \end{align}
 It is easier to show that $(\mathrm{VI})$ tends to 0 as $n\to\infty$ and then $l\to\infty$.
 Thanks to the Kipnis--Varadhan inequality (Proposition \ref{KVI}),
 \begin{equation}
   (\mathrm{VI})\leq C_0\int_0^T \left\lVert\frac{1}{\sqrt{n}}\sum_{j\in\mathbb{Z}}\phi\!\left(s,\frac{y_j}{n} \right) L^{(n)}_{B_j}f_j\right\rVert_{-1,n}^2\,ds.
 \end{equation}
 By definition of the $\|\cdot\|_{-1,n}$ norm and the fact that
 \begin{equation}
   \left<g,\left(-n^2L_n \right)g \right>_{\nu^{(n)}_\rho}\geq \sum_{j\in\mathbb{Z}} \left<g,\left(-n^2L^{(n)}_{B_j} \right)g \right>_{\nu^{(n)}_\rho},
 \end{equation}
 we obtain
 \begin{equation}
 \begin{split}
   &(\mathrm{VI})\\
   &\leq C\sum_{j\in\mathbb{Z}}\int_0^T\sup_{g\,:\,local} \left\{2\left<\frac{1}{\sqrt{n}}\phi\!\left(s,\frac{y_j}{n} \right) L^{(n)}_{B_j}f_j,g \right>_{\nu^{(n)}_\rho} - \left<g,\left(-n^2L^{(n)}_{B_j} \right)g \right>_{\nu^{(n)}_\rho} \right\}\,ds\\
   &\leq C\sum_{j\in\mathbb{Z}}\int_0^T\frac{1}{n^3}\phi\!\left(s,\frac{y_j}{n} \right)^2 \sup_{g\,:\,local} \left\{2\left< L^{(n)}_{B_j}f_j,g \right>_{\nu^{(n)}_\rho} - \left<g,\left(-L^{(n)}_{B_j} \right)g \right>_{\nu^{(n)}_\rho} \right\}\,ds\\
   &\leq C\sum_{j\in\mathbb{Z}}\int_0^T \frac{1}{n^3}\phi\!\left(s,\frac{y_j}{n} \right)^2 \left<f_0, \left(-L^{(n)}_{B_0} \right)f_0 \right>_{\nu^{(n)}_\rho}\,ds.
 \end{split}
 \end{equation}
  The last inequality holds since the supremum in the integrand is attained by $g=-f_j$ and $\nu^{(n)}_\rho$ is translation invariant. Therefore,
  \begin{equation}
    (\mathrm{VI})\leq \frac{C}{n^2} \left<f_0, \left(-L^{(n)}_{B_0} \right)f_0 \right>_{\nu^{(n)}_\rho}.
  \end{equation}
  Thus we obtain the statement \eqref{VI}. Note that we used Lemma \ref{ev_cvgc} to establish the boundedness of $\left<f_0,\left(-L^{(n)}_{B_0} \right)f_0 \right>_{\nu^{(n)}_\rho}$.
  
  Finally, let us show \eqref{V}. Using the Cauchy--Schwarz inequality and the fact that $\sum_{x\in B^0_i} \tau_x V-L^{(n)}_{B_{i}}f_i$ and $\sum_{y\in B^0_j} \tau_y V-L^{(n)}_{B_{j}}f_j$ are independent as long as $i\neq j$,
  \begin{equation}
    \begin{split}
      &(\mathrm{V})\\
      &\leq C\mathbb{E}_n\!\left[\int_0^T \frac{1}{n}\left\{\sum_{j\in\mathbb{Z}}\phi\!\left(s,\frac{y_j}{n} \right) \left(-L^{(n)}_{B_j}f_j(\eta(s))+\sum_{x\in B^0_j} \tau_xV^{(n)} (\eta(s)) \right) \right\}^2 \,ds\right]\\
      &\leq \frac{C}{n}\int_0^T\sum_{j\in\mathbb{Z}}\phi\!\left(s,\frac{y_j}{n} \right)^2 \mathbb{E}_n\!\left[\left(-L^{(n)}_{B_j}f_j(\eta(s))+\sum_{x\in B^0_j} \tau_xV^{(n)} (\eta(s)) \right)^2 \right]\,ds\\
      &\leq \frac{C}{l} E_{\nu^{(n)}_\rho}\!\left[\left(-L^{(n)}_{B_0}f_0+\sum_{x\in B^0_0} \tau_xV^{(n)} \right)^2 \right]\\
      &= \frac{C}{l} E_{\nu^{(\infty)}_\rho}\!\left[\left(-L^{(n)}_{B_0}f_0+\sum_{x\in B^0_0} \tau_xV^{(n)} \right)^2\frac{d\nu^{(n)}_\rho}{d\nu^{(\infty)}_\rho} \right].\\
    \end{split}
  \end{equation}
  Letting $n\to\infty$ and using the lemmas in Section \ref{section-phi-D},\, $0\leq r_n \leq r_\infty\in L^2(\nu^{(\infty)}_\rho)$ and the dominated convergence theorem,  we get
  \begin{equation}\label{inf_f_0}
    \inf_{f_0} \limsup_{n\to\infty} (\mathrm{V})\leq \frac{C}{l} \inf_{f_0}E_{\nu^{(\infty)}_\rho}\!\left[\left(-L^{(\infty)}_{B_0}f_0+\sum_{x\in B^0_0} \tau_xV^{(\infty)} \right)^2 \right].
  \end{equation}
  On the right-hand side we can see the orthogonal projection onto the orthogonal complement of the range of $L^{(\infty)}_{B_0}$ which is regarded as an operator on $L^2(\mathbb{N}^l, \nu^{(\infty)}_\rho )$. Hereafter we are working on $L^2(\mathbb{N}^{l}, \nu^{(\infty)}_\rho )$ in this section.

  Let $RL^{(\infty)}_{B_0}$ be the range of $L^{(\infty)}_{B_0}$, $(RL^{(\infty)}_{B_0})^{\perp}$ be its orthogonal complement and $\eta_0^{l}\coloneq\frac{1}{l}\sum_{x=0}^{l-1}\eta_x$. Then we can prove that
  \begin{equation}\label{range_perp}
    (RL^{(\infty)}_{B_0})^{\perp} = \left\{f\in L^2(\mathbb{N}^l, \nu^{(\infty)}_\rho )\,:\, \text{values of $f$ depend only on } \eta_0^l \right\}.
  \end{equation}
  Indeed, we can easily check that the left-hand side includes the right-hand side. For the converse inclusion, take $f\in(RL^{(\infty)}_{B_0})^{\perp}$. Then 

  \begin{equation}
    \left<f, -L^{(\infty)}_{B_0}f \right>_{\nu^{(\infty)}_\rho} = E_{\nu^{(\infty)}_\rho}\!\left[\sum_{\substack{x,y\in B_0\\ |x-y|=1}} r_\infty(\eta_x,\eta_y)\left\{f(\eta^{x,y})-f(\eta) \right\}^2 \right]=0.
  \end{equation}
  Therefore we obtain \eqref{range_perp}.

  Now we are ready to prove \eqref{inf_f_0}. Let $E_{\nu^{(\infty)}_\rho}[\,\cdot\,|\eta_0^l]$ denote the conditional expectation with respect to the $\sigma$-algebra generated by $\eta_0^l$. Then \eqref{inf_f_0} becomes
  \begin{equation}\label{condi_ex}
    \inf_{f_0} \limsup_{n\to\infty} (\mathrm{V})\leq \frac{C}{l} E_{\nu^{(\infty)}_\rho}\!\left[E_{\nu^{(\infty)}_\rho}\!\left[\sum_{x\in B^0_0} \tau_xV^{(\infty)} \,\middle|\,\eta_0^l\,\right]^2 \right].
  \end{equation}
  In the above formula, 
  \begin{equation}
    E_{\nu^{(\infty)}_\rho}\!\left[\sum_{x\in B^0_0} \tau_xV^{(\infty)} \,\middle|\,\eta_0^l\right] = \sum_{x\in B^0_0} \left\{E_{\nu^{(\infty)}_\rho}\!\left[\tau_x f\,\middle|\,\eta_0^l \right] - \Phi^{(\infty)}_f(\rho) - D^{(\infty)}_f(\rho)(\eta_0^l-\rho) \right\}.
  \end{equation}
  Here we use $E_{\nu^{(\infty)}_\rho}\!\left[\eta_x \,\middle|\,\eta_0^l\,\right]=\eta_0^l$. 

  \subsubsection{Step 3: cut off at large density}
  
  It is necessary to obtain a good bound of the right-hand side of \eqref{condi_ex}.

  Let $L$ be a positive number which is supposed to diverge after $l\to\infty$.
  \begin{equation}
    \frac{1}{l} E_{\nu^{(\infty)}_\rho}\!\left[E_{\nu^{(\infty)}_\rho}\!\left[\sum_{x\in B^0_0} \tau_xV^{(\infty)} \,\middle|\,\eta_0^l\,\right]^2 \right]\leq3\left[(\mathrm{VII})+(\mathrm{VIII})+(\mathrm{IX}) \right],
  \end{equation}
  where
  \begin{align*}
    (\mathrm{VII})&\coloneq\frac{1}{l}  E_{\nu^{(\infty)}_\rho}\!\left[\left\{\sum_{x\in B_0^0}\left(E_{\nu^{(\infty)}_\rho}\!\left[\tau_x f\,\middle|\,\eta_0^l \right] - \Phi^{(\infty)}_{\tau_xf}(\eta_0^l) \right) \right\}^2\mathbbm{1}_{\{\eta_0^l\leq L\}}  \right]\\
    (\mathrm{VIII})&\coloneq\frac{1}{l}  E_{\nu^{(\infty)}_\rho}\!\left[\left\{\sum_{x\in B_0^0}\left(\Phi^{(\infty)}_{\tau_xf}(\eta_0^l)- \Phi^{(\infty)}_f(\rho) - D^{(\infty)}_f(\rho)(\eta_0^l-\rho)\right) \right\}^2\mathbbm{1}_{\{\eta_0^l\leq L\}}  \right]  \\
    (\mathrm{IX})&\coloneq\frac{1}{l}  E_{\nu^{(\infty)}_\rho}\!\left[E_{\nu^{(\infty)}_\rho}\!\left[\sum_{x\in B^0_0} \tau_xV^{(\infty)} \,\middle|\,\eta_0^l\,\right]^2\mathbbm{1}_{\{\eta_0^l\geq L\}} \right].\\
  \end{align*}
  It remains to estimate each term.
  Firstly, thanks to the equivalence of ensembles \cite[Appendix 2, Corollary 1.7]{Kipnis_Landm},
  \begin{equation}
  \begin{split}
    (\mathrm{VII})&\leq C\sum_{x\in B^0_0} E_{\nu^{(\infty)}_\rho}\!\left[\left\{E_{\nu^{(\infty)}_\rho}\!\left[\tau_x f\,\middle|\,\eta_0^l \right] - \Phi^{(\infty)}_{\tau_xf}(\eta_0^l) \right\}^2\mathbbm{1}_{\{\eta_0^l\leq L\}} \right]\\
    &\leq C(L)\sum_{x\in B^0_0} \frac{1}{l^2}\\
    &\leq \frac{C(L)}{l}.
  \end{split}
  \end{equation}
  which tends to 0 as $l\to\infty$. Here $C(L)$ denotes generic constants which may depend on $L$.

  Furthermore, by the Taylor expansion of $\Phi^{(\infty)}_{f}$ around $\rho$, 
  \begin{equation}
    \begin{split}
      (\mathrm{VIII})&\leq C(L)\sum_{x\in B^0_0}E_{\nu^{(\infty)}_\rho}\!\left[(\eta_0^l-\rho)^4 \right]\\
      &\leq \frac{C(L)}{l}
    \end{split}
  \end{equation}
  which also tends to 0 as $l\to\infty$. Here we used the fact that $\nu^{(\infty)}_\rho$ is a product measure and that $E_{\nu^{(\infty)}_\rho}\!\left[(\eta_0^l-\rho)^4 \right]=O(l^{-2})$.

  For $(\mathrm{IX})$, by Jensen's inequality and translation invariance of $\nu^{(\infty)}_\rho$,
  \begin{equation}
    \begin{split}
      (\mathrm{IX})\leq CE_{\nu^{(\infty)}_\rho}\!\left[V^{(\infty)}(\eta)^2\mathbbm{1}_{\{\eta_0^l\geq L\}} \right].
    \end{split}
  \end{equation}
  Note that the right-hand side does not depend on $l$ and tends to 0 as $L\to\infty$. This completes the proof of Theorem \ref{BGP}.

  \section{Derivation of the SHE}
  The rest of this manuscript is devoted to the proof of Theorem \ref{Main-theorem}. We use the Dynkin martingale.
  \begin{align*}
    \mathcal{M}_t^n(\phi)\coloneq \mathcal{X}_t^n(\phi)-\mathcal{X}_0^n(\phi) -\int_0^t (\partial_s+n^2L_n)\mathcal{X}_s^n(\phi)\,ds.
  \end{align*}
  It is well-known that its quadratic variation equals
  \begin{equation}
  \begin{split}
    &\left<\mathcal{M}^n(\phi) \right>_t\\
    &= \int_0^t\left(n^2L_n\left(\mathcal{X}_s^n(\phi)^2 \right)-2\mathcal{X}_s^n(\phi)\,n^2L_n\mathcal{X}_s^n(\phi) \right)\,ds\\
    &=\int_0^t \frac{1}{n}\sum_{x\in\mathbb{Z}}\left\{r_n(\eta_x(s),\eta_{x+1}(s))+r_n(\eta_x(s),\eta_{x-1}(s))\right\} \left(\nabla^n\phi(s,\frac{x}{n}) \right)^2\,ds,
  \end{split}
  \end{equation}
  where $\nabla^n$ denotes $\nabla^n\phi(s,\frac{x}{n})\coloneq n\left\{\phi(s,\frac{x+1}{n})-\phi(s,\frac{x}{n})\right\}$. Note that, for each $t\in[0,T]$, the right-hand side of the above formula converges almost surely to $2\Phi_c^{(\infty)}(\rho)\int_0^t\lVert\nabla\phi(s,\cdot)\rVert_{L^2(\mathbb{R})}^2\,ds$ as $n\to\infty$.

  In addition, a simple computation shows that
  \begin{equation}
    n^2L_n\eta_x=n^2\left\{c(\eta_{x+1})+c(\eta_{x-1})-2c(\eta_x) \right\}.
  \end{equation}
  Plugging this into the definition of the Dynkin martingale and using summation by parts,
  \begin{equation}
    n^2L_n\mathcal{X}_s^n(\phi)=\frac{1}{\sqrt{n}}\sum_{x\in\mathbb{Z}} c(\eta_x(s))\Delta_n\phi\!\left(s,\frac{x}{n}\right),
  \end{equation}
  here $\Delta_n$ denotes the discrete Laplacian $\Delta_n\phi\!\left(s,\frac{x}{n}\right)=n^2\left(\phi\!\left(s,\frac{x+1}{n}\right)+\phi\!\left(s,\frac{x-1}{n}\right)-2\phi\!\left(s,\frac{x}{n}\right) \right)$.
  Therefore we have
  \begin{align*}
    \mathcal{M}_t^n(\phi)=\mathcal{X}_t^n(\phi)-\mathcal{X}_0^n(\phi)-\int_0^t\mathcal{X}_s^n(\partial_s\phi)+ \frac{1}{\sqrt{n}}\sum_{x\in\mathbb{Z}}\left(c(\eta_x(s))-\Phi_c^{(n)}(\rho) \right)\Delta_n\phi\!\left(s,\frac{x}{n}\right)\,ds.
  \end{align*}
  Above, we used the fact that the telescoping sum vanishes, so that we can freely recenter the variable $c(\eta)$.

  At a heuristic level, the term $c(\eta_x(s))-\Phi_c^{(n)}(\rho)$ can be approximated by $D^{(\infty)}_c(\rho)\left(\eta^l_x(s)-\rho \right)$ thanks to the Boltzmann--Gibbs principle. Furthermore, as we shall see below, the process $\mathcal{X}^n$ has a limit in the Skorokhod topology, along some subsequence which we still denote by $n$ by abuse of notation. Consequently, setting the limiting process $\mathcal{X}_\cdot\coloneq\lim_{n\to\infty}\mathcal{X}^{n}_\cdot$, we can see that
  \begin{align*}
    \lim_{n\to\infty}\mathcal{M}_t^n(\phi)&= \mathcal{X}_t(\phi)-\mathcal{X}_0(\phi)-\int_0^t\mathcal{X}_s(\partial_s\phi+D^{(\infty)}_c(\rho)\Delta\phi)\,ds\\
    \lim_{n\to\infty}\left<\mathcal{M}^n(\phi) \right>_t&=2\Phi_c^{(\infty)}(\rho)\int_0^t\lVert\nabla\phi(s,\cdot)\rVert_{L^2(\mathbb{R})}^2\,ds
  \end{align*}
  In other words, if we can prove the above formulae and that $\lim_{n\to\infty}\mathcal{M}_t^n(\phi)$ is a martingale with quadratic variation $\lim_{n\to\infty}\left<\mathcal{M}^n(\phi) \right>_t$, then $\mathcal{X}_\cdot$ is a solution to the SHE. This is the outline of the proof of the main theorem.
  
  \subsection{Tightness}
  In this section we prove the tightness of the distribution of $\mathcal{X}^n_\cdot $ on $D([0,T]:\mathcal{S}'(\mathbb{R})) $. By Mitoma's criterion (see \cite{Mitoma1983TightnessOP}), it suffices to show the tightness of $\left\{\mathcal{X}^n_\cdot(\phi) \right\}_n$ for any $\phi\in\mathcal{S}(\mathbb{R})$.

  By the Dynkin martingale decomposition, we have

  \begin{equation}
    \mathcal{X}^n_t(\phi)=\mathcal{X}^n_0(\phi)+\mathcal{M}^n_t(\phi)+\mathcal{S}^n_t(\phi)+\mathcal{R}^n_t(\phi),
  \end{equation}
  where
  \begin{align*}
    \mathcal{S}^n_t(\phi)&\coloneq \int_{0}^t \frac{1}{\sqrt{n}}\sum_{x\in\mathbb{Z}}\left(c\!\left(\eta_x(s)\right)-\Phi_c^{(n)}(\rho) \right)\left(\partial_s +\Delta\right)\phi\!\left(s,\frac{x}{n}\right)\,ds\\
    \mathcal{R}^n_t(\phi)&\coloneq \int_0^t \frac{1}{\sqrt{n}}\sum_{x\in\mathbb{Z}}\left(c\!\left(\eta_x(s)\right)-\Phi_c^{(n)}(\rho) \right)\left\{\Delta_n\phi\!\left(s,\frac{x}{n}\right)-\Delta\phi\!\left(s,\frac{x}{n}\right) \right\}\,ds. \\
  \end{align*}
  By the above decomposition, it suffices to show that $\left\{\mathcal{X}^n_0(\phi) \right\}_n$, $\left\{\mathcal{M}^n_\cdot(\phi) \right\}_n$ and $\left\{\mathcal{S}^n_\cdot(\phi) \right\}_n$ are tight and that
  \begin{equation}\label{tightness-R}
    \lim_{n\to\infty}\mathbb{E}_n\!\left[\sup_{t\in[0,T]}\lvert \mathcal{R}^n_t(\phi)\rvert^2 \right]=0 .
  \end{equation} 

  For $\left\{\mathcal{X}^n_0(\phi) \right\}_n$, it can be checked that their characteristic functions converge to that of a Gaussian random variable.
  
  \subsubsection{The remainder part}
  Let us prove \eqref{tightness-R}. Since we have the Taylor expansion,
  \begin{equation}
    \begin{split}
      \left\{\Delta_n\phi\!\left(s,\frac{x}{n}\right)-\Delta\phi\!\left(s,\frac{x}{n}\right) \right\}\leq C(\phi)\frac{1}{n^2}\left\{\frac{1}{1+\left(\frac{x}{n} \right)^2}+\frac{1}{1+\left(\frac{x+1}{n} \right)^2} \right\}.
    \end{split}
  \end{equation}
  Moreover, since $\phi\in C([0,T]: \mathcal{S}(\mathbb{R}))$ and we have the Cauchy--Schwarz inequality,

  \begin{equation}
    \begin{split}
      &\mathbb{E}_n\!\left[\sup_{t\in[0,T]}\lvert \mathcal{R}^n_t(\phi)\rvert^2 \right] \\
      &\leq C\int_0^T \frac{1}{n}\mathbb{E}_n\!\left[\sum_{x\in\mathbb{Z}}\left(c\!\left(\eta_x(s)\right)-\Phi_c^{(n)}(\rho) \right)\left\{\Delta_n\phi\!\left(s,\frac{x}{n}\right)-\Delta\phi\!\left(s,\frac{x}{n}\right) \right\}\right]^2\,ds\\
      &\leq C\int_0^T \frac{1}{n}\sum_{x\in\mathbb{Z}}\sum_{y\in\mathbb{Z}}E_{\nu^{(n)}_\rho}\!\left[\left(c(\eta_x)-\Phi^{(n)}_c(\rho) \right)\left(c(\eta_y)-\Phi^{(n)}_c(\rho) \right) \right]\\
      &\quad\quad\quad \frac{1}{n^2}\left\{\frac{1}{1+\left(\frac{x}{n} \right)^2}+\frac{1}{1+\left(\frac{x+1}{n} \right)^2} \right\}\frac{1}{n^2}\left\{\frac{1}{1+\left(\frac{y}{n} \right)^2}+\frac{1}{1+\left(\frac{y+1}{n} \right)^2} \right\}\,ds\\
      &\leq C\frac{1}{n^5}\sum_{x\in\mathbb{Z}}\text{Var}_{\nu^{(n)}_\rho}\!\left[c(\eta_x) \right] \left\{\frac{1}{1+\left(\frac{x}{n} \right)^2}+\frac{1}{1+\left(\frac{x+1}{n} \right)^2} \right\}^2\\
      &\leq C\frac{1}{n^4}\to 0 \quad \text{as $n\to\infty$}.
    \end{split}
  \end{equation}
  Note that $\text{Var}_{\nu^{(n)}_\rho}\!\left[c(\eta_x) \right]$ is bounded in $n$.

  \subsubsection{The martingale part}
  In order to prove the tightness of $\left\{\mathcal{M}^n_\cdot\right\}$, we use the following criterion given in \cite[Theorem VIII.3.12]{jacod2013limit}.
  \begin{proposition}
    Let $\left\{M^n_\cdot \right\}_{n\in\mathbb{N}} $ be a sequence of martingales
    belonging to the space $D\left([0,T]:\mathbb{R} \right)$ and denote by $\left<M^n \right>_\cdot$ the quadratic variation of $M^n_\cdot$.
    Let $q\colon[0,T]\to\mathbb{R}_{+}$ be a deterministic continuous function.
    Assume that
    \begin{enumerate}
      \item For any $n\in\mathbb{N} $, the quadratic variation process $\left\{\left<M^n\right>_t:t\in[0,T] \right\}$ has continuous trajectories almost surely.
      \item The following limit holds:
        \begin{equation}
          \lim_{n\to\infty}\mathbb{E}\!\left[\sup_{0\leq t \leq T}\lvert M^n_t- M^n_{t-}\rvert \right]=0.
        \end{equation}
      \item For any $t\in[0,T]$, the sequence of random variables $\left\{\left<M^n \right>_t \right\}_n$ converges to $q(t)$ in probability.
    \end{enumerate}
    Then, the sequence $\left\{M^n_\cdot \right\}_{n\in\mathbb{N}} $ converges in distribution on $D\left([0,T]:\mathbb{R} \right)$
    to a mean-zero Gaussian process $\left\{M_t:t\in[0,T] \right\}$, which is a martingale with continuous trajectories and whose quadratic variation is given by $q$.
  \end{proposition}
  We apply the above proposition to the martingale $\mathcal{M}^n_\cdot(\phi)$ and 
  \begin{equation}
    q(t)\coloneq 2\Phi_c^{(\infty)}(\rho)\int_0^t\lVert\nabla\phi(s,\cdot)\rVert_{L^2(\mathbb{R})}^2\,ds.
  \end{equation}
  The first item follows immediately. 
  
  For the second item, recall that the probability that two or more jumps occur simultaneously is zero. Hence we have that
  \begin{equation}
    \begin{split}
      \sup_{0\leq t \leq T}\lvert \mathcal{M} ^n_t(\phi)- \mathcal{M} ^n_{t-}(\phi)\rvert&\leq\sup_{0\leq t \leq T}\lvert \mathcal{X} ^n_t(\phi)- \mathcal{X} ^n_{t-}(\phi)\rvert \\
      &\leq \sup_{t\in[0,T]}\sup_{x\in\mathbb{Z}}\left| \frac{1}{\sqrt{n}}\phi\!\left(t,\frac{x+1}{n}-\frac{x}{n}\right)\right|\\
      &\leq C\frac{1}{n^{\frac{3}{2}}}.
    \end{split}
  \end{equation}

  For the final item, we are going to prove that
  \begin{equation}
    \lim_{n\to\infty} \mathbb{E}_n\!\left[\left| \left<\mathcal{M}^n(\phi) \right>_t-q(t) \right| \right]=0.
  \end{equation}
  Since it can be checked that
  \begin{equation}
    \lim_{n\to\infty}\mathbb{E}_n\!\left[\left| \int_0^t \frac{1}{n}\sum_{x\in\mathbb{Z}}\frac{c(\eta_x(s))c(\eta_{x+1}(s))}{c(\alpha_n)}\left(\nabla^n\phi\left(s,\frac{x}{n}\right) \right)^2\,ds \right| \right]=0,
  \end{equation}
  it suffices to prove that
  \begin{equation}
    \lim_{n\to\infty}\mathbb{E}_n\!\left[\left| \int_0^t \frac{1}{n}\sum_{x\in\mathbb{Z}}[c(\eta_x(s))+c(\eta_{x+1}(s))]\left(\nabla^n\phi\!\left(s,\frac{x}{n}\right) \right)^2\,ds-q(t) \right| \right]=0.
  \end{equation}
  Moreover, we can replace $\nabla^n$ with $\nabla$ in the above formula by the Taylor expansion. 
  In other words, our task is to show that
  \begin{equation}
    \lim_{n\to\infty}\mathbb{E}_n\!\left[\left| \int_0^t \frac{1}{n}\sum_{x\in\mathbb{Z}}[c(\eta_x(s))+c(\eta_{x+1}(s))]\left(\nabla\phi\!\left(s,\frac{x}{n}\right) \right)^2\,ds-q(t) \right| \right]=0.
  \end{equation}
  In the above formula,
  \begin{equation}
    \mathbb{E}_n\!\left[\left| \int_0^t \frac{1}{n}\sum_{x\in\mathbb{Z}}[c(\eta_x(s))+c(\eta_{x+1}(s))]\left(\nabla\phi\!\left(s,\frac{x}{n}\right) \right)^2\,ds-q(t) \right| \right]\leq 3((\mathrm{I})+(\mathrm{II})+(\mathrm{III})).
  \end{equation}
  where 
  \begin{align*}
    (\mathrm{I})&\coloneq \mathbb{E}_n\!\left[\left| \int_0^t \frac{1}{n}\sum_{x\in\mathbb{Z}}[c(\eta_x(s))-\Phi^{(n)}_c(\rho)]\left(\nabla\phi\!\left(s,\frac{x}{n}\right) \right)^2\,ds\right| \right]\\
    (\mathrm{II})&\coloneq \mathbb{E}_n\!\left[\left| \int_0^t \frac{1}{n}\sum_{x\in\mathbb{Z}}[c(\eta_{x+1}(s))-\Phi^{(n)}_c(\rho)]\left(\nabla\phi\!\left(s,\frac{x}{n}\right) \right)^2\,ds\right| \right]\\
    (\mathrm{III})&\coloneq 2\Phi^{(n)}_c(\rho)\left | \int_{0}^{t}\frac{1}{n}\sum_{x\in\mathbb{Z}}\left(\nabla\phi\!\left(s,\frac{x}{n}\right) \right)^2\,ds - \int_{0}^{t} \int_{\mathbb{R}}\left(\nabla\phi\!\left(s,u\right) \right)^2 \,duds \right |
  \end{align*}
  It is not difficult to see $\lim_{n\to\infty}(\mathrm{III})=0$.

  For $(\mathrm{I})$ and $(\mathrm{II})$, we apply the Cauchy--Schwarz inequality and use the fact that $\nu^{(n)}_\rho$ is a product measure. 

  \subsubsection{The symmetric part}
  By the Cauchy--Schwarz inequality,
  \begin{equation}
    \begin{split}
      \mathbb{E}_n\!\left[ \left| \mathcal{S}^n_t(\phi)-\mathcal{S}^n_s(\phi) \right|^2 \right]&\leq C |t-s| \int_{s}^{t}\frac{1}{n}\sum_{x\in\mathbb{Z}}E_{\nu^{(n)}_\rho}\!\left[ c(\eta_x)^2 \right]\!\left(\left(\partial_s+\Delta \right) \phi\!\left(s,\frac{x}{n}\right)  \right)^2\,ds\\
      &\leq C|t-s|^2.
    \end{split}
  \end{equation}
  Then we can apply the Kolmogorov--Chentsov criterion (see \cite[Theorem~1.2.1]{RevusYor}) to prove the tightness of $\left\{\mathcal{S}^n_\cdot(\phi)\right\}_n$ and to show thatthe limiting process has continuous paths almost surely.

  \subsection{Identification of limit points}
  So far we proved that the sequences $\{\mathcal{X}^n\}_n$, $\{\mathcal{M}^n\}_n$, $\{\mathcal{S}^n\}_n$ are tight. Let $\mathcal{Q}^n$ be the distribution of the tuple $\{\mathcal{X}^n, \mathcal{M}^n, \mathcal{S}^n\}$. Then we can take a subsequence of it (still denoted by $\{\mathcal{Q}^n\}_n$) that converges in distribution to some probability measure $\mathcal{Q}$ on $D([0,T]: \mathcal{S}'(\mathbb{R}))^3$.
  In addition, let $\mathcal{X}_\cdot, \mathcal{M}_\cdot, \mathcal{S}_\cdot$ be the respective limiting processes on $D([0,T]: \mathcal{S}'(\mathbb{R}))$.

  Since each of the processes in the martingale decomposition has continuous paths almost surely, $\mathcal{Q} $ is concentrated on $C([0,T]: \mathcal{S}'(\mathbb{R}))$. Thus the following projection is continuous $\mathcal{Q}$-almost surely:
  \begin{equation}
    \pi_\phi:D([0,T]: \mathcal{S}'(\mathbb{R}))\ni \mathcal{X}_\cdot \mapsto \mathcal{X}_\cdot(\phi)\in D([0,T]:\mathbb{R}).
  \end{equation}
  Therefore, passing to the limit as $n\to\infty$ on both sides of the following
  \begin{equation}
    \mathcal{M}^n_t(\phi)= \mathcal{X}^n_t(\phi)-\mathcal{X}^n_0(\phi)-\mathcal{S}^n_t(\phi),
  \end{equation}
  we obtain
  \begin{equation}
    \mathcal{M}_t(\phi)= \mathcal{X}_t(\phi)-\mathcal{X}_0(\phi)-\mathcal{S}_t(\phi).
  \end{equation}
  Here, by the definition of $\mathcal{S}^n_t(\phi)$,
  \begin{equation}
    \mathcal{S}^n_t(\phi)=\int_0^t \mathcal{X}^n_s(\partial_s\phi_s)\,ds+\int_{0}^{t}\frac{1}{\sqrt{n}}\sum_{x\in\mathbb{Z}}\left(c(\eta_x(s))-\Phi^{(n)}_c(\rho) \right)\Delta\phi\!\left(s,\frac{x}{n} \right)\,ds.
  \end{equation}
  Thanks to Theorem \ref{BGP}, 
  \begin{equation}
    \int_{0}^{t}\frac{1}{\sqrt{n}}\sum_{x\in\mathbb{Z}}\left(c(\eta_x(s))-\Phi^{(n)}_c(\rho) \right)\Delta\phi\!\left(s,\frac{x}{n} \right)\,ds= \int_{0}^{t}\mathcal{X}^n_s\!\left(D^{(\infty)}_c(\rho)\Delta\phi_s\right)\,ds + \tilde{\mathcal{R}}_t^{n}(\phi),
  \end{equation}
  where $\tilde{\mathcal{R}}_t^{n}(\phi)$ is a real-valued random variable satisfying
  \begin{equation}
    \lim_{n\to\infty}\mathbb{E}_n\!\left[\sup_{t\in[0,T]}\lvert \tilde{\mathcal{R}}_t^{n}(\phi)\rvert^2 \right]=0.
  \end{equation}
  Therefore, since $\mathcal{X}^n_\cdot(\phi)$ converges to $\mathcal{X}_\cdot(\phi)$ in distribution and $\mathcal{X}_\cdot(\phi)$ has continuous paths, we have
  \begin{equation}
    \mathcal{S}_t(\phi)=\int_{0}^{t}\mathcal{X}_s\!\left((\partial_s+D^{(\infty)}_c(\rho)\Delta)\phi_s\right)\,ds,
  \end{equation}
  which implies that $\mathcal{X}$ is a solution to the SHE. This completes the proof of Theorem \ref{Main-theorem} by uniqueness in law of the SHE. 

  \appendix
  \section{Preliminaries for the stochastic heat equation}\label{sec-SHE}
  In this appendix, we recall some basic notions concerning Schwartz spaces, tempered distributions, white noise, and $\mathcal{S}'(\mathbb{R})$-valued Brownian motions that are used in the formulation of the stochastic heat equation. See \cite{DaPratoZabczyk} and \cite{hairer2023introductionstochasticpdes} for further details.
  \subsection{Schwartz space}
  The Schwartz space plays a central role in the theory of tempered distributions. It is used in Section \ref{sec-main-result} to define the solution of the stochastic heat equation.
  \begin{definition}
    The Schwartz space is the function space
    \begin{equation}
      \mathcal{S}(\mathbb{R}^n)\coloneq \left\{f\in C^\infty(\mathbb{R}^n)\,:\, \forall m\in\mathbb{N}_0,\, \|f\|_m<\infty \right\},
    \end{equation}
    where the seminorm $\|\cdot\|_m$ is given by
    \begin{equation}
      \|f\|_m\coloneq \sup_{\substack{\alpha=(\alpha_1,\cdots,\alpha_n)\in\mathbb{N}_0^n\\ \alpha_1+\cdots+\alpha_n\leq m}}\, \sup_{x\in\mathbb{R}^n} (1+|x|^2)^\frac{m}{2} |\partial^\alpha f(x)|.
    \end{equation}
    The space $\mathcal{S}(\mathbb{R}^n)$ is equipped with the seminorm topology induced by $\{\|\cdot\|_m \}_{m\in\mathbb{N}_0}$. It is known that $\mathcal{S}(\mathbb{R}^n)$ is a Polish space.
  \end{definition}

  \subsection{Tempered distributions with the strong topology}
  The topological dual of $\mathcal{S}(\mathbb{R}^n)$ is called the space of tempered distributions and is denoted by $\mathcal{S}'(\mathbb{R}^n)$. We introduce two well-known topologies on the space of tempered distributions $\mathcal{S}'(\mathbb{R}^n)$: the weak-$*$ topology and the strong topology.
  \begin{definition}
    The weak-$*$ topology is the coarsest topology on $\mathcal{S}'(\mathbb{R}^n)$ which makes the evaluation map $\mathcal{S}'(\mathbb{R}^n)\ni f\mapsto \langle f, \phi \rangle \in\mathbb{R}$ continuous for all $\phi\in \mathcal{S}(\mathbb{R}^n)$.
  \end{definition}
  In other words, the weak-$*$ topology is closed under the pointwise convergence.

  \begin{definition}
    The strong topology is the coarsest topology on $\mathcal{S}'(\mathbb{R}^n)$ which makes the following map $p_K\colon\mathcal{S}'(\mathbb{R}^n)\ni f\mapsto \sup_{\phi\in K}|\langle f, \phi \rangle|\in\mathbb{R}$ continuous for all compact subsets $K \subset \mathcal{S}(\mathbb{R}^n)$.
  \end{definition}

  Intuitively, the strong topology is the appropriate one when one wishes to consider uniform convergence on compact sets. Hereafter, we equip $\mathcal{S}'(\mathbb{R}^n)$ with the strong topology.

  \subsection{White noise and $\mathcal{S}'(\mathbb{R})$-valued Brownian motion}
  \begin{definition}
    An $\mathcal{S}'(\mathbb{R})$-valued random variable $\mathcal{W}$ is called the white noise with mean 0 and variance $\chi$ if for any $l\in\mathbb{N}$ and any $\phi_1,\phi_2,\cdots,\phi_l\in\mathcal{S}(\mathbb{R})$, $\left\{\mathcal{W}(\phi_1),\cdots,\mathcal{W}(\phi_l) \right\}$ is a mean-zero Gaussian random vector with covariances given by
    \begin{equation}
      E\!\left[\mathcal{W}(\phi_i)\mathcal{W}(\phi_j) \right]=\chi \langle\phi_i,\phi_j\rangle_{L^2(\mathbb{R})}.
    \end{equation}
  \end{definition}
  Next we consider the space $C([0,T]:\mathcal{S}'(\mathbb{R}))$ with the topology defined below.
  \begin{definition}
    $C([0,T]:\mathcal{S}'(\mathbb{R}^n))$ is the space of continuous functions from $[0,T]$ to $\mathcal{S}'(\mathbb{R}^n)$ equipped with the topology induced by a family of seminorms $q_K$ indexed by compact subsets $K \subset \mathcal{S}(\mathbb{R}^n)$.
    \begin{equation}
      q_K\colon C([0,T]:\mathcal{S}'(\mathbb{R}^n))\ni f\mapsto \sup_{t\in[0,T]}p_K\!\left(f(t)\right)\in\mathbb{R}.
    \end{equation}
  \end{definition}

  \begin{fact}[see e.g. {\cite[Proposition 34.6]{Treves}}]
    $\{y_t\}_{t\in[0,T]}$ belongs to $C([0,T]:\mathcal{S}'(\mathbb{R}))$ if and only if $\{y_t(\phi)\}_{t\in[0,T]}$ belongs to $C([0,T]:\mathbb{R})$ for all $\phi\in\mathcal{S}(\mathbb{R})$.
  \end{fact}

  \begin{definition}
    We say that $\{\mathcal{B}_t\}_{t\in[0,T]}$ is an $\mathcal{S}'(\mathbb{R})$-valued Brownian motion if $\{\mathcal{B}_t\}_{t\in[0,T]}$ belongs to $C([0,T]:\mathcal{S}'(\mathbb{R}))$ and $\{\lVert\phi\rVert_{L^2(\mathbb{R})}^{-1}\mathcal{B}_t(\phi)\}_{t\in[0,T]}$ is a Brownian motion on $\mathbb{R}$ for all $\phi\in\mathcal{S}(\mathbb{R})\backslash \{0\}$.
  \end{definition}
\nocite{komorowski2012fluctuations}
\nocite{JARA20205973}
\nocite{JaraFlores2019}
\nocite{HAYASHI2023180}
\nocite{DGP2017}
\nocite{liggett2025continuous}
\nocite{liggett1985interacting}
\nocite{hairer2023introductionstochasticpdes}
\bibliographystyle{amsplain} 
\bibliography{references} 
\end{document}